%% file: main.tex
\pdfoutput=1
\documentclass[mnsc,nonblindrev]{informs_modified}

\OneAndAHalfSpacedXI

\usepackage{natbib}
 \bibpunct[, ]{(}{)}{,}{a}{}{,}%
 \def\bibfont{\small}%
\usepackage{color}
\usepackage{array}
\usepackage{multirow}
\usepackage{pifont,xspace}
\usepackage{amsmath,bbm}
\usepackage{caption}
\usepackage{subcaption}
\usepackage[super]{nth}

\newcolumntype{L}{>{\centering\arraybackslash}m{3cm}}
\newcommand{\PROOF}[1]{\noindent \textbf{Proof of #1.}}

\newcommand{\cN}{\mathcal{N}}
\newcommand{\bR}{\mathbb{R}}

\newcommand{\bI}{\mathbf{1}}

\newcommand{\ls}{\lambda^{(s)}}
\newcommand{\lsd}{\lambda^{(s+\delta)}}
\newcommand{\inv}{^{-1}}

\renewcommand{\vec}[1]{\boldsymbol{#1}}

\newcommand{\UCBlin}{\mathsf{OFUL}}
\newcommand{\AdaptedOFUL}{\mathsf{SemiUCB}}
\newcommand{\UCBours}{\mathsf{ConsUCB}}

\newcommand{\da}{\text{\ding{172}}\xspace}
\newcommand{\db}{\text{\ding{173}}\xspace}
\newcommand{\dc}{\text{\ding{174}}\xspace}
\newcommand{\dd}{\text{\ding{175}}\xspace}
\newcommand{\de}{\text{\ding{176}}\xspace}

\usepackage{comment}

\TheoremsNumberedThrough     
\ECRepeatTheorems

\EquationsNumberedThrough    

\MANUSCRIPTNO{}

\begin{document}

\RUNAUTHOR{Jin et al.}

\RUNTITLE{Large Scale Product Selection}

\TITLE{Shrinking the Upper Confidence Bound: A Dynamic Product Selection Problem for Urban Warehouses}

\ARTICLEAUTHORS{%
\AUTHOR{Rong Jin}
\AFF{Alibaba Group, San Mateo, CA 94402, \EMAIL{jinrong.jr@alibaba-inc.com}}
\AUTHOR{David Simchi-Levi}
\AFF{Institute for Data, Systems, and Society, Department of Civil and Environmental Engineering, and Operations Research Center, Massachusetts Institute of Technology, Cambridge, MA 02139, \EMAIL{dslevi@mit.edu}}
\AUTHOR{Li Wang}
\AFF{Operations Research Center, Massachusetts Institute of Technology, Cambridge, MA 02139, \EMAIL{li\_w@mit.edu}}
\AUTHOR{Xinshang Wang}
\AFF{Institute for Data, Systems, and Society, Massachusetts Institute of Technology, Cambridge, MA 02139, \EMAIL{xinshang@mit.edu}}
\AUTHOR{Sen Yang}
\AFF{Alibaba Group, San Mateo, CA 94402, \EMAIL{senyang.sy@alibaba-inc.com}}
}

\ABSTRACT{
The recent rising popularity of ultra-fast delivery services on retail platforms fuels the increasing use of urban warehouses, whose proximity to customers makes fast deliveries viable.
The space limit in urban warehouses poses a problem for such online retailers: the number of products (SKUs) they carry is no longer ``the more, the better'', yet it can still be significantly large, reaching hundreds or thousands in a product category.
In this paper, we study algorithms for dynamically identifying a large number of products (i.e., SKUs) with top customer purchase probabilities on the fly, from an ocean of potential products to offer on retailers' ultra-fast delivery platforms.

We distill the product selection problem into a semi-bandit model with linear generalization.
There are in total $N$ arms, each with a feature vector of dimension $d$.
The player pulls $K$ arms in each period and observes the bandit feedback from each of the pulled arms.
We focus on the setting where $K$ is much greater than the number of total time periods $T$ or the dimension of product features $d$.
We first analyze a standard UCB algorithm and show its regret bound can be expressed as the sum of a $T$-independent part $\tilde O(K d^{3/2})$ and a $T$-dependent part $\tilde O(d\sqrt{KT})$, which we refer to as ``fixed cost'' and ``variable cost'' respectively.
To reduce the fixed cost for large $K$ values, we propose a novel online learning algorithm, which iteratively shrinks the upper confidence bounds within each period, and show its fixed cost is reduced by a factor of $d$ to $\tilde O(K \sqrt{d})$.
Moreover, we test the algorithms on an industrial dataset from Alibaba Group. Experimental results show that our new algorithm reduces the total regret of the standard UCB algorithm by at least 10\%.
}


\maketitle

\section{Introduction}
\label{sec:intro}
\input{sec-01-intro}

\section{Model Formulation}
\label{sec:model}

\input{sec-02-model}

\section{Alternative Analysis of a Standard UCB Algorithm}
\label{sec:UCB1}

\input{sec-03-ucb1}

\section{Online Learning with Conservative Exploration}
\label{sec:UCB2}
\input{sec-04-ucb2}

\section{Numerical Experiments}
\label{sec:numerical}
\input{sec-05-numerical}

\section{Conclusion and Insights for Supply Chain Managers}
\label{sec:conclusion}

\input{sec-06-conclusion}

\newpage

\begin{APPENDIX}{Proof of Lemma \ref{lm:Xinshang_lemma}}

\label{appendix:proofs}

\input{sec-08-appendixUCB1}

\end{APPENDIX}

\bibliographystyle{ormsv080}
\bibliography{myrefs}

\end{document}

%% file: sec-01-intro.tex
In this paper, we study a large-scale product selection problem, motivated by the rising popularity of ultra-fast delivery on retail platforms.
With ultra-fast delivery services such as same-day or instant delivery, customers receive parcels within hours after placing online orders, whereas traditional deferred delivery usually takes one to five days.
The demand for ultra-fast delivery is strong: a survey conducted in 2016 shows that 25\% of customers are willing to pay significant premiums for the service \citep{McKinsey}.
On the supply side, many online retailers' ``spin-off'' platforms that offer ultra-fast delivery have emerged in the market, for example, Amazon Prime Now, Walmart Jet.com, Ocado, and Alibaba Hema.
As a result, the retail market for ultra-fast delivery has enjoyed exponential growth over the past few years.
In the United States, the total order value of same-day delivery merchandise reached 4.03 billion dollars in 2018, up from 0.1 billion dollars in 2014 \citep{statista}.

The rapid growth in the ultra-fast delivery retail market is accompanied by the increasing use of urban warehouses \citep{JLL}.
Their proximity to customers is the key in making ultra-fast delivery viable for retailing models, especially in high-traffic cities like New York City and San Francisco.
Online retailers typically integrate the urban warehouses for their fast-delivery ``spin-off'' platforms into their existing supply chains, by having them replenished by large suburban distribution centers that originally support their ``traditional'' online retail platforms only.
However, due to the limited space in urban warehouses, the retail platforms offering ultra-fast delivery carry fewer products than the traditional online platforms that are directly fulfilled by distribution centers.
Hence, retailers often face the challenge of selecting a good subset of different products (i.e., SKUs) to offer on the ultra-fast delivery platforms, from an ocean of potential products that are carried in the distribution centers.

Thanks to modern inventory-management technologies, such as Kiva robots, the cost of maintaining a diversified inventory is no more than that of managing an inventory with only a few types of products \citep{QZ}.
Thus, online retailers offering ultra-fast delivery are disposed to carry more SKUs than the set of most popular items such as the market-leader or mainstream products with major market shares.
By increasing product breadth, they gain an edge over local grocery stores or supermarkets, in terms of satisfying customers' sporadic demands for products with medium-to-low popularities.

For example, as of October 2018, Prime Now, Amazon's two-hour delivery retail platform, offers around 1,500 different products in the toys category at zip code 02141 (Cambridge, MA).
Meanwhile, a local CVS store at the same zip code only carries around 100 popular toys, whereas the traditional online retail platform Amazon.com includes more than 90,000 SKUs related to toys.
Table \ref{table:1} summarizes product coverages of different retailers, where products are partitioned into two types: 1) market-leader, and 2) medium/low-demand.
The focus in this paper is on selecting an optimal set of products with medium-to-low demands for online retailers with urban warehouses.

\begin{table}[h!]
\centering
\begin{tabular}{|c|c|c|c|}
\hline
 & \multirow{2}{*}{Brick and mortar} & \multicolumn{2}{c|}{Online retail platforms}\\
 \cline{3-4}
 & & Urban warehouse & Distribution center \\
\hline
Market-leader products & Most/All & All & All  \\
\hline
Medium/low-demand products& None/Few  & \textbf{Some (paper focus)} & Most  \\
\hline
\end{tabular}
\caption{Table of product assortment coverages of different retailers}
\label{table:1}
\end{table}

For urban warehouse retailers, the most popular items can be readily identified from their sales volumes.
However, with their total number of SKUs in each category reaching hundreds, if not thousands, it is a more difficult task to accurately estimate the popularities of a large number of relatively low-demand products.
Fortunately, online retailers are capable of adjusting the product sets in their urban warehouses, by exchanging a part of the inventory with their distribution centers during regular replenishments.
This enables the retailers to learn each product's probability of being sold through dynamic product offering, which in turn makes the problem a more complicated task of sequential decision-making than a pure demand estimation problem.

In this paper, we study different algorithms for dynamically selecting an optimal set of products for online retailers that use urban warehouses to support their ultra-fast delivery services.
The algorithms learn the demand levels of all products on the fly, and give retailers the ability to have their product sets tailored to every city, or even every zip code, better catering to customers' geodemographic preferences.
We distill the product selection problem into a semi-bandit model with linear generalization. For this model, we first provide an alternative analysis of a popular existing algorithm, which we call $\AdaptedOFUL$, that is adapted from the algorithm $\UCBlin$ proposed by \citet{NIPS2011_4417}.
Then, we propose a novel algorithm called $\UCBours$, and show it is superior in both theoretical and numerical performances, especially when the retailer selects products and observes their sales in large batches.

\subsection{Model Overview}

We consider an online retailer who wants to find an optimal set of products (SKUs) in each category to be offered on her ultra-fast delivery retail platform in a certain city area.
In the rest of the paper, we use ``product'' and ``SKU'' interchangeably.

Due to the use of urban warehouses, the storage space, after including the market-leader products identified from past sales data, is limited.
In each product category, assuming products take similar storage space, the retailer would like to select additional $K$ SKUs from a catalog of $N$ products that have medium-to-low customer demands.
Because of their relatively low popularities, there is a very limited amount of available sales data related to the candidate products in that city.
Therefore, the retailer is interested in learning the optimal set of products through experimentation.
Specifically, in each period over a $T$-period sales time horizon, she selects $K$ products to offer and observes whether customers purchase those products, which helps her learn the product demands more accurately and then adjust the product set accordingly in the next periods. 

By only considering products with relatively low demands in this problem (see the discussion of product coverage summarized in Table \ref{table:1}), we assume their sales are independent.
The reason is that the demands for the $N$ candidate products are typically direct demands, in the sense that customers specifically search for these products for some particular reasons.
In the case where some products are not included in the offered product set, their demands, if not lost, 
are often captured by the market-leader products.
Therefore, the probability of shifting from one low-popularity product to another is very low.

We assume the candidate products' probabilities of realizing positive sales in each period are an unknown linear function of their $d$-dimensional feature vectors that are known a priori and fixed over time.
Specifically, for any product with feature vector $x \in \bR^d$, its positive-sales probability in any period is $x^{\top}\theta^*$, where $\theta^*$ is unknown.

\begin{figure}[h]
    \centering
    \includegraphics[scale=0.7]{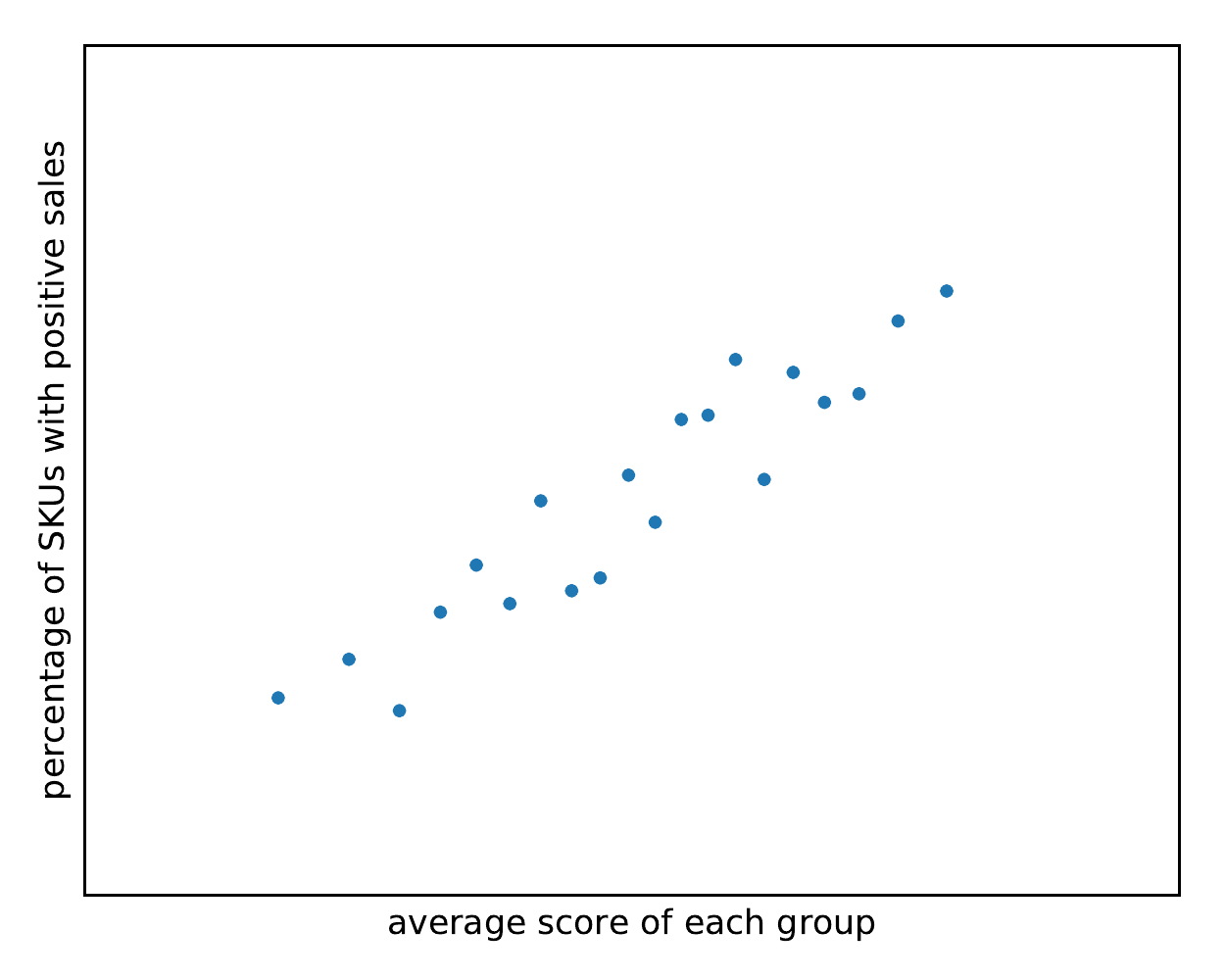}
    \caption{Sales data from Alibaba Tmall suggest a near-linear relationship between products' probabilities of positive sales and their scores, which are linear in product features.}
    \label{fig:linearModel}
\end{figure}

We examine the linear model assumption between products' features and their probabilities of positive sales on industrial data from Alibaba Tmall, the largest business-to-consumer online retail platform in China.
We select around $N\approx 20,000$ products in the toys category with medium-to-low popularities in a city region.
Each product is associated with a $50$-dimensional feature vector, which is derived from the product's intrinsic features as well as product-related user activities, like clicks, additions to cart, and purchases, by Alibaba's deep learning team using representation learning \citep{Bengio2013}.
Two consecutive weeks of the region's unit sales data are translated into binary sales data indicating whether a product had positive sales in each week.

To test the validity of the linear assumption, we first fit a linear model for the first-week binary sales with the product features as covariates.
Then, for all products, we compute their scores based on the feature vectors using the fitted model, and divide the products into groups of size around 1,000 according to their scores in ascending order.
The positive-sales probabilities are estimated by the arithmetic averages of the products' second-week binary sales for all groups.
In Figure~\ref{fig:linearModel}, the estimated positive-sales probabilities are plotted against the average scores for all 20 groups, and a near-linear relationship is clearly identified.
Since the scores are linear in the features, this suggests that a linear relationship between the products' features and  probabilities of positive sales can be reasonably assumed.

Since fast-delivery retailing is a rapid-growing business, such retailers' objective in the early stages after launching the platforms is typically focused on user growth and retention, by providing customers with good shopping experiences.
Therefore, in this problem, we assume the retailer's goal is to maximize the number of products with positive sales, which represent successes in matching customers' demands.
Then, the expected reward in each period is the sum of positive-sales probabilities of the offered products.
We measure an algorithm's performance by its regret, the loss in total expected reward compared to the optimal expected reward.

Such a model is usually referred to as \emph{semi-bandits with linear generalization}.
In the literature, more general models like \emph{contextual combinatorial semi-bandits} have been studied mainly for different applications such as personalized movie recommendations or online advertising \citep{Yue2011,doi:10.1137/1.9781611973440.53,Wen:2015:ELL:3045118.3045237}.
In those applications, the number of recommendations $K$ is typically very small compared to the number of periods $T$.
Indeed, \citet{Yue2011} explicitly make the assumption that $K \leq d$ (recall that $d$ is the dimension of the feature vector space).

In the task of choosing optimal product sets for urban-warehouse retailers, however, the number of products selected in every period is usually very large ($K\geq 1000$ products), whereas the learning process typically needs to be within a quarter ($T\leq 26$, for semiweekly updates), to reduce operations cost and potential long-term consumer confusion.
Motivated by this problem setting, we focus on a different regime of problem parameters in the semi-bandit model with linear generalization, where the number of selections $K$ is large but the number of periods $T$ is small.

Directly applying existing algorithms and analyses in this setting results in regret bounds that grow rapidly in $K$ (see the discussion of regret bounds of different algorithms summarized in Table \ref{table:2}).
Therefore, this raises the following research question that we aim to address in this work:
\begin{quote}
\textit{How can a retailer optimize the exploration-exploitation tradeoff in selecting optimal product sets from a pool of $N \geq 20000$ products, when there are only a few time periods ($T\leq26$) available for experimentation, but she is able to observe sales information for a large number of products ($K \geq1000$) in each period?}
\end{quote}
Regarding whether this task is achievable, we point out that, even though the number of periods $T$ is small, a good algorithm can still effectively learn the true model parameter $\theta^*$ on the fly, because there are a significant number $KT$ of sales observations after all.

Since the time horizon we consider is no more than three months, we assume the set of candidate products remain unchanged.
Moreover, products' intrinsic characteristics and their related user behaviors are in general unlikely to have major changes during the relatively short time horizon.
Hence, we assume that the product features are fixed over time in our problem setting.
It is worth noting that, such assumptions can be removed once the time horizon is over, given $\theta^*$ is accurately estimated in the end.

For the purpose of presenting a clean model, we defer more discussions of the model's connection to practice to later sections. For instance, in Section 5, we provide a numerical study on the consequence of not considering any shipping constraints when adjusting the product sets in the model. In Section 6, we explain other possible variants of the model.

\subsection{Main Contributions}

In addition to numerically testing algorithm performances on an online retailer's data in Section~\ref{sec:numerical}, we make two main technical contributions in this paper:
\begin{enumerate}
	\item We provide an alternative analysis of a common Upper Confidence Bound algorithm, which we call $\AdaptedOFUL$, in our model setting (Section~\ref{sec:UCB1}). The main idea behind $\AdaptedOFUL$ is to use an ellipsoid to construct the confidence region for $\theta^*$. This technique has been studied in various semi-bandit models (c.f. \citet{Yue2011, doi:10.1137/1.9781611973440.53}). We contribute to the literature by proving an alternative regret bound for this algorithm when $K$ is large.
	 
Our new regret bound of $\AdaptedOFUL$ can be expressed as a combination of two parts:
	(i) The first part is $\tilde O(d\sqrt{KT})$, which is sub-linear in both $K$ and $T$ ($\tilde O(\cdot)$ hides logarithmic factors). We call this part the ``variable cost'', as it increases with the length of time horizon $T$. The variable cost is a standard regret term in linear contextual bandit problems after playing $KT$ arms (selecting $KT$ products).
	(ii) The second part is $\tilde O(K d^{3/2})$, which is linear in $K$ and independent of $T$. We call it the ``fixed cost'', since it is independent of the length of the time horizon. The fixed cost is due to the unobservable feedback within selecting $K$ products in each period.
	
	
We also show that the fixed cost of $\AdaptedOFUL$ is at least $\Omega(K\min(d,T))$. From the business point of view, this lower bound is discouraging, because over the $T \leq 26$ periods, the total regret $\Omega(K\min(d,T))$ of the algorithm could be of the same order as the total reward $O(K T)$. In other words, the standard UCB technique and its immediate extensions could lead to arbitrarily bad performance over the entire season.  
This motivates us to devise new exploration-exploitation techniques to beat this lower bound on the fixed cost.

	\item	In order to improve on the fixed cost of $\AdaptedOFUL$, we propose a novel ``conservative'' confidence bound that shrinks within each period, and propose a new algorithm $\UCBours$ (Section~\ref{sec:UCB2}). Using the conservative confidence bound and selecting $K$ products sequentially within each period, the new algorithm intelligently takes advantage of the abundance of sales observations in each period and enjoys an improved fixed cost.
	
The regret bounds proved in this paper are summarized and compared with existing results in Table \ref{table:2}.
Although the models in \citet{Yue2011, Wen:2015:ELL:3045118.3045237, doi:10.1137/1.9781611973440.53} are more general than ours, their regret bounds often involve terms such as $K \sqrt{d T}$ and are thus not suitable when $K$ is large. By contrast, in the regret bound of $\UCBours$, the fixed cost that is linear in $K$ is only $\tilde O(K\sqrt{d})$. 

It is also worth noting that the $\sqrt{d}$ factor in the fixed cost of $\UCBours$ is due to linear generalization. In practice, this $\sqrt{d}$ factor and other logarithmic factors are replaced by a parameter that is tuned to achieve good empirical performance (see Section \ref{sec:numerical}). Despite these factors related to linear generalization, the fixed cost of $\UCBours$ is only $O(K)$. For online retailers, it implies that the fixed cost of $\UCBours$ is no more than the optimal reward over only a few periods. 
On the other hand, there is a simple $\Omega(K)$ lower bound on the fixed cost of \emph{any algorithm}, since an $\Omega(K)$ regret in the \emph{first time period} is unavoidable. Therefore, the fixed cost of $\UCBours$ is optimal except for a $\tilde O(\sqrt{d})$ factor caused by linear generalization.


\begin{table}[h!]
\centering
\small
\begin{tabular}{|c|c|c|c|}
\hline
\textbf{Paper} & \textbf{Algorithm technique} & \begin{tabular}{cc} \textbf{``Fixed cost''}  \\ \textbf{regret}\end{tabular} & \begin{tabular}{cc} \textbf{``Variable cost''}  \\ \textbf{regret}\end{tabular}\\
\hline
\citet{Yue2011} & UCB & - & $\tilde O(K \sqrt{d T} + d\sqrt{K T})$\\
\hline
\citet{Wen:2015:ELL:3045118.3045237} & Thompson sampling & - & $\tilde O(K \sqrt{dT \min(\ln N, d)})$\\
\hline
\citet{Wen:2015:ELL:3045118.3045237} & UCB & - & $\tilde O(K d \sqrt{T})$\\
\hline
\citet{doi:10.1137/1.9781611973440.53} & UCB  & - & $\tilde O(K \sqrt{d T} + d\sqrt{K T})$\\
\hline
this paper (Section~\ref{sec:UCB1})&  \begin{tabular}{cc} UCB  \\ (standard algorithm $\AdaptedOFUL$)\end{tabular} & $\tilde O(K d^{3/2})$ & $\tilde O(d \sqrt{KT})$\\
\hline
 this paper (Section~\ref{sec:UCB2})& \begin{tabular}{cc}UCB with conservative exploration \\ (new algorithm $\UCBours$)\end{tabular}& $\tilde O(K\sqrt{d})$ & $\tilde O(d \sqrt{KT})$\\
\hline
\end{tabular}
\caption{Regret bounds of algorithms adapted to our model setting where $K$ is much larger than $T$ or $d$.}
\label{table:2}
\end{table}

\end{enumerate}

	In the final part of the paper, we use Alibaba's data to test and compare the performances of $\AdaptedOFUL$ and $\UCBours$. Numerically, we show that $\UCBours$ reduces about $10\%$ of the regret of $\AdaptedOFUL$ in the first ten periods.

\subsection{Related Literature}

Our model is a type of multi-armed bandit problem, if we view the $N$ products as $N$ distinct arms. In classic problems of multi-armed bandit, a player sequentially pulls arms without initially knowing which arm returns the highest reward in expectation. The player needs to update the strategy of pulling arms based on the bandit feedbak of the pulled arm in each round, in order to minimize the total regret over time. For a broader review on multi-armed bandit problems, we refer the reader to \citet{Bubeck2012, Slivkins2017}.

There are two special characteristics of our model. First, we assume the reward of pulling an arm is a linear function of an embedding feature vector of the arm. Second, in each step the player is able to pull a very large number of different arms and then observe the bandit feedback for \emph{each} of the pulled arms.

In the literature, multi-armed bandit models assuming linear reward/payoff functions are often referred to as linear contextual bandits. \citet{auer2002using, DHK08, doi:10.1287/moor.1100.0446, pmlr-v15-chu11a, NIPS2011_4417, Bubeck2012b, AgrawalG12} propose and analyze different algorithms for linear models in which the player is able to pull only \emph{one arm} in each period. The sampling algorithm in \citet{RV14a} also applies to this linear setting. Notably, the $\AdaptedOFUL$ algorithm that we analyze in Section \ref{sec:UCB1} is a direct extension of the OFUL algorithm in \citet{NIPS2011_4417}. We are able to provide a new type of analysis of $\AdaptedOFUL$ for our special model (see Section \ref{sec:UCB1}), in which the player can pull \emph{a large number} of different arms (each  arm is associated with a feature vector) in each period.

When the reward of each arm is a linear function of covariates that are drawn i.i.d. from a diverse distribution, \citet{Goldenshluger2013, Bastani, MikeWei} propose online learning algorithms whose regret bounds are proved to be poly-log in $T$. However, their regret bounds scale at least linearly with the total number of arms $N$. By contrast, the regret bounds of our algorithms are independent of $N$.

When the player can pull multiple arms in each period and observe the bandit feedback from each of the pulled arms, the model is often referred to as semi-bandits \citep{Audibert2012}. In the literature, most semi-bandit models view the player's action as a vector in $\{0,1\}^N$ (c.f. \citet{Cesa-Bianchi2012, Gai2012, pmlr-v28-chen13a}); they do not further assume the reward of each arm $i \in [N]$ to be a linear function. 

To our knowledge, only \citet{Yue2011}, \citet{Gabillon2014}, \citet{doi:10.1137/1.9781611973440.53}, \citet{Wen:2015:ELL:3045118.3045237} study semi-bandit models in which the reward of each arm is generalized to a linear function. The models in these papers are more general than ours, as they allow for combinatorial constraints or submodular reward functions. However, their research focuses on applications in which $K$ is much smaller than $T$. Thus, their regret bounds grow rapidly in $K$. The regret bounds in \citet{Yue2011}, \citet{doi:10.1137/1.9781611973440.53}, \citet{Wen:2015:ELL:3045118.3045237} contain $O(K\sqrt{dT})$. The regret bound in \citet{Gabillon2014} contains $O(K \sum_i \frac{1}{\Delta_i})$, where $\Delta_i$ are gaps between sub-optimal and optimal arms. By contrast, the regret bound of our $\UCBours$ algorithm is $\tilde O(K \sqrt{d} + d\sqrt{KT})$, in which the term that is linear in $K$ is only $\tilde O(K\sqrt{d})$.

%% file: sec-02-model.tex
Throughout the paper, we use $[k]$ to denote the set $\{1,2,\ldots,k\}$ for any positive integer $k$.

There are $N$ distinct products with low-to-medium customer demands.
There are $T$ periods, and in each period $t\in[T]$, the retailer selects a set of $K$ products, denoted as $S_t$, from $[N]$ to offer on her retail platform.

Each product $i\in[N]$ has a feature vector $x_i\in\mathbb{R}^d$ that is known to the retailer in advance and stays fixed over time.
The probability of positive sales of product $i$ in any period is denoted as $\mu(i)$, which is linear in its feature vector $x_i$, i.e., $\mu(i)=x_i^\top\theta^*$ for some unknown vector $\theta^*\in\mathbb{R}^d$.
We assume $\Vert \theta^*\Vert\leq 1$ and $\mu(i) \in [0,1]$, $\Vert x_{i}\Vert\leq 1$ for each product $i\in[N]$.

In each period $t$, the binary random variable $r_{t,i}\in\{0,1\}$ denotes whether the realized sales of product $i$ in period $t$ are positive, for each product $i$ in the selected product set $S_t$.
We assume $r_{t,i}$ is independent across products and across time periods, and its expected value $\mathbb{E}[r_{t,i}]=\mu(i)=x_i^\top\theta^*$.

The expected reward in each period is the sum of positive sales probabilities of the offered products, and the $T$-period total expected reward is $\sum_{t=1}^T\sum_{i\in S_t}\mu(i)$.
The optimal product set $S^* \in\argmax\limits_{S\subset[N],|S|=K}\sum_{i\in S}\mu(i)$ is a set of $K$ products with the highest probabilities of positive sales.
The performance of any online algorithm is measured by its regret, which is defined as
\[
R(T)=\sum_{t=1}^T\sum_{i\in S^*}\mu(i)-\sum_{t=1}^T\sum_{i\in S_t}\mu(i).
\]

In each period $t \in [T]$, the retailer makes the decision $S_t$ based on all the past information including $\{r_{t',i}\}_{t' =1,2,\ldots,t-1; i \in S_{t'}}$. The goal is to minimize the total regret over the $T$ time periods.

%% file: sec-03-ucb1.tex
In this section, we focus on a popular existing UCB algorithm, which we call $\AdaptedOFUL$, that has been widely used in practice and analyzed in theory.


We provide an alternative analysis of $\AdaptedOFUL$ and prove a new regret bound, which consists of a $T$-independent ``fixed cost'' $\tilde{O}(Kd^{3/2})$ and a $T$-dependent ``variable cost'' $\tilde{O}(d\sqrt{KT})$.
This alternative analysis allows us to more accurately evaluate the regret terms for $\AdaptedOFUL$, especially in our model where $K$ is significantly larger than $T$.

We also give an example to show an $\Omega(K \min(d,T))$ lower bound on the regret of $\AdaptedOFUL$, which illustrates the algorithm's potential weakness in some cases.
This lower bound implies that, in the regret bound of $\AdaptedOFUL$, the fixed cost is at least $\Omega(K d)$, which is much more significant than the variable cost $\tilde{O}(d\sqrt{KT})$ in our problem setting.
Hence, the idea of modifying $\AdaptedOFUL$ to reduce the fixed cost leads to the development of the new algorithm $\UCBours$ in Section~\ref{sec:UCB2}.

\subsection{A Standard UCB Algorithm for Semi-Bandits with Linear Generalization}
$\AdaptedOFUL$ is adapted from the standard UCB algorithm $\UCBlin$ designed for linear contextual bandits \citep{NIPS2011_4417}.
The difference is that, for $\AdaptedOFUL$, the model is updated once $K$ products (arms) are offered (pulled) in each period, whereas, for $\UCBlin$, the model is updated every time one product (arm) is offered (pulled) in each period.
If $K$ is set to $1$, the two algorithms become the same.
$\AdaptedOFUL$ is presented step by step in the following part.

\
\\
$\AdaptedOFUL$ algorithm for semi-bandits with linear generalization (with input parameters $\alpha, \omega$):
\begin{enumerate}
    \item Initialize $A_0= \omega I_{d\times d}$ and $b_0=\vec{0}_{d}$.
    \item Repeat for $t=1,2,3,\ldots,T$
    \begin{itemize}
        \item[\ \ ] (a) Set $\theta_t= A_{t-1}^{-1}b_{t-1}$.
        \item[\ \ ] (b) Calculate $p_{t}(i)=x_{i}^\top\theta_t+\alpha\sqrt{x_{i}^\top A_{t-1}^{-1}x_{i}}$, for all $i=1,2,3,\ldots,N$.
        \item[\ \ ] (c) Offer product set $S_t\in\argmax\limits_{S\subset\cN,|S|=K}\left\{\sum_{i\in S}p_{t}(i)\right\}$.
        Observe outcomes $r_{t,i}\in\{0,1\}$ for $i\in S_t$.
        \item[\ \ ] (d) Update $A_t= A_{t-1}+\sum_{i\in S_t}x_{i}x_{i}^\top$ and $b_t= b_{t-1}+\sum_{i\in S_t}r_{t,i}x_{i}$.
    \end{itemize}
\end{enumerate}

\

The algorithm $\AdaptedOFUL$ we present above is an extension of $\UCBlin$ \citep{NIPS2011_4417}, and it can also be considered as a special case of its combinatorial version $\mathsf{C^2UCB}$ \citep{doi:10.1137/1.9781611973440.53}. From \citet{doi:10.1137/1.9781611973440.53}, we know that if $\AdaptedOFUL$ is run with $\alpha=\sqrt{d\log{\left(\frac{1+TN/K}{\delta}\right)}}+\sqrt{K}$ and $\omega = K$, then the regret of the algorithm is $\tilde O(K \sqrt{d T} + d\sqrt{K T})$ with probability at least $1 - \delta$.

\subsection{New Regret Bound for a Standard UCB Algorithm}
In this section, we provide an alternative analysis of the regret of $\AdaptedOFUL$, when applied to semi-bandit models with linear generalization.

We prove the regret of $\AdaptedOFUL$ is $\tilde{O}(Kd^{3/2}+d\sqrt{KT})$, which is a sum of two parts.
The first part $\tilde{O}(Kd^{3/2})$ is largely due to the model's inability to observe product sales feedback within selecting $K$ products in each period, and is shown to be independent of the number of time periods, $T$.
The second part $\tilde{O}(d\sqrt{KT})$ is a common regret term for UCB-type algorithms in linear contextual bandit problems.

If we consider the regret terms as ``costs'' that an algorithm has to pay in the learning process, the first part of the regret resembles an ``fixed cost'', as it does not increase with $T$, while the second part is similar to a ``variable cost'', as it increases with $T$.

\subsubsection{Existing results related to linear contextual bandits.} We first introduce some existing results in the literature that we will use later on. For any positive definite matrix $A\in\mathbb{R}^{d \times d}$, we define the weighted $2$-norm of any vector $x\in\mathbb{R}^{d}$ as
\[
\lVert x\rVert_{A}=\sqrt{x^\top Ax}.
\]

Recall that, for each product $i\in[N]$, $\mu(i)$ is $x_i^\top\theta^*$ and $p_{t}(i)$ is defined as $x_{i}^\top\theta_t+\alpha\sqrt{x_{i}^\top A_{t-1}^{-1}x_{i}}$ for each period $t\in[T]$ in $\UCBlin$.
\begin{lemma}[{\citet{doi:10.1137/1.9781611973440.53}, Lemma 4.1}]
\label{lm:qin_lem_4.1}
If we run $\AdaptedOFUL$ with $\alpha = \sqrt{d\log\left(\frac{1+TN}{\delta}\right)}+1$ and $\omega = 1$, then we have, with probability at least $1-\delta$, for all periods $t\in[T]$ and all products $i\in[N]$,
\[
0\leq p_{t}(i)-\mu(i)\leq 2\alpha\lVert x_{i}\rVert_{A^{-1}_{t-1}}.
\]
\end{lemma}

For each period $t \in [T]$, let $x_{(t,1)}, x_{(t,2)},\ldots, x_{(t,K)}$ denote an arbitrary permutation of the $K$ feature vectors $\{x_i : i \in S_t\}$.

Define
\[
A_{t,k}=A_{t-1}+\sum_{i=1}^k x_{(t,i)}x_{(t,i)}^\top,
\]
for each period $t\in[T]$ and each $k\in[K]\cup \{0\}$. The following lemma is a direct result of Lemma~3 in \citet{pmlr-v15-chu11a}, by assuming the $KT$ product selections are made in a model where only one product needs to be selected in each period for a total of $KT$ periods.

\begin{lemma}[{\citet{pmlr-v15-chu11a}, Lemma 3}] If we run $\AdaptedOFUL$ with $\omega = 1$, then
\label{lm:chu_lem_3}
\[
\sum_{t=1}^T
\sum_{k=1}^K\lVert x_{(t,k)}\rVert_{A^{-1}_{t,k-1}}
\leq
5\sqrt{dKT\log(KT)}
.
\]
\end{lemma}

\subsubsection{New analysis for $\AdaptedOFUL$.} The next lemma is the key lemma of our analysis for $\AdaptedOFUL$. It upper-bounds the difference between two norms of the same vector $x$ weighted by two matrices, $A$ and $A+\sum_{k=1}^L u_k u_k^\top$. The additional sum of $L$ outer products corresponds to updating matrix $A$ using $L$ feature vectors selected by $\AdaptedOFUL$. We defer its proof to the appendix.

\begin{lemma}\label{lm:Xinshang_lemma}
Let $A \in \mathbb{R}^{d\times d}$ be any symmetric positive definite matrix, and $u_1, u_2,...,u_L \in \mathbb{R}^d$ be any vectors.
Let $\lambda_1, \lambda_2, \ldots, \lambda_d$ be the eigenvalues of $A$, and $\nu_1,\nu_2,\ldots, \nu_d$ be the eigenvalues of $A + \sum_{k=1}^L u_k u_k^\top$.
We have, for any $x \in \mathbb{R}^d$ such that $\|x\|_2 =1$,
\[
\sqrt{x^\top A^{-1} x} - \sqrt{x^\top (A +  \sum_{k=1}^L u_k u_k^\top)^{-1} x} \leq  \sum_{i=1}^d \frac{2}{\sqrt{\lambda_i}} - \sum_{i=1}^d \frac{2}{\sqrt{\nu_i}}.
\]
\end{lemma}

Our new regret bound for $\AdaptedOFUL$ is presented in the following theorem.
\begin{theorem}\label{thm:UCB1}
If $\AdaptedOFUL$ is run with
\[
\alpha=\sqrt{d\log{\left(\frac{1+TN}{\delta}\right)}}+1, \quad \omega = 1,
\]
then with probability at least $1-\delta$, the regret of the algorithm is
\[ \tilde{O}(Kd^{3/2} + d \sqrt{KT}).\]
\end{theorem}

\proof{Proof.}
By Lemma~\ref{lm:qin_lem_4.1} and the property that $\AdaptedOFUL$ picks products with the largest UCB values $p(i)$, we have, with probability at least $1-\delta$,
\begin{align*}
R(T)
&=
\sum_{t=1}^T\sum_{i\in S^*}\mu(i)
-
\sum_{t=1}^T\sum_{i\in S_t}\mu(i)
\\
&\leq
\sum_{t=1}^T\sum_{i\in S^*}p_{t}(i)
-
\sum_{t=1}^T\sum_{i\in S_t}\mu(i)
\\
&\leq
\sum_{t=1}^T\sum_{i\in S_t}p_{t}(i)
-
\sum_{t=1}^T\sum_{i\in S_t}\mu(i)
\\
&\leq
\sum_{t=1}^T\sum_{i\in S_t}2\alpha\lVert x_{i}\rVert_{A^{-1}_{t-1}}
\\
&=
2\alpha\sum_{t=1}^T\sum_{i\in S_t}\lVert x_{i}\rVert_{A^{-1}_{t-1}}
.
\end{align*}

Recall that $x_{(t,1)},\ldots, x_{(t,K)}$ is a sequence of feature vectors of products in $S_t$.
Moreover, $A_{t,k}=A_{t-1}+\sum_{i=1}^kx_{(t,i)}x_{(t,i)}^\top$, for each period $t\in[T]$ and each $k\in[K]\cup\{0\}$.
Then, we can continue to obtain
\begin{align}
R(T) & \leq 2\alpha\sum_{t=1}^T\sum_{i\in S_t}\lVert x_{i}\rVert_{A^{-1}_{t-1}}\nonumber \\
&=
2\alpha\sum_{t=1}^T
\left(
\sum_{i\in S_t}\lVert x_{i}\rVert_{A^{-1}_{t-1}}
-\sum_{k=1}^K\lVert x_{(t,k)}\rVert_{A^{-1}_{t,k-1}}
+\sum_{k=1}^K\lVert x_{(t,k)}\rVert_{A^{-1}_{t,k-1}}
\right)\nonumber
\\
&=
2\alpha\sum_{t=1}^T
\left(
\sum_{k=1}^K\lVert x_{(t,k)}\rVert_{A^{-1}_{t-1}}
-\sum_{k=1}^K\lVert x_{(t,k)}\rVert_{A^{-1}_{t,k-1}}
\right)
+2\alpha\sum_{t=1}^T
\sum_{k=1}^K\lVert x_{(t,k)}\rVert_{A^{-1}_{t,k-1}}\nonumber
\\
&=
2\alpha\sum_{t=1}^T
\sum_{i=1}^K
\left(
\lVert x_{(t,k)}\rVert_{A^{-1}_{t-1}}
-
\lVert x_{(t,k)}\rVert_{A^{-1}_{t,k-1}}
\right)
+
2\alpha\sum_{t=1}^T
\sum_{k=1}^K\lVert x_{(t,k)}\rVert_{A^{-1}_{t,k-1}}
.\label{eq:proofUCB1a}
\end{align}

For the second term in \eqref{eq:proofUCB1a}, we have by Lemma~\ref{lm:chu_lem_3},
\begin{align}\label{eq:UCB1_part1}
\begin{split}
2\alpha
\sum_{t=1}^T
\sum_{k=1}^K\lVert x_{(t,k)}\rVert_{A^{-1}_{t,k-1}}
&\leq 2 \left(\sqrt{d\log{\left(\frac{1+TN}{\delta}\right)}}+1\right) \cdot 5\sqrt{dKT\log(KT)}\\
&=
10d\sqrt{KT\log\left(\frac{1+KT}{\delta}\right)\log(KT)}
+
10\sqrt{dKT\log(KT)}
.
\end{split}
\end{align}

Let $\lambda_{t,1}, \lambda_{t,2}, \ldots, \lambda_{t,d}$ be the eigenvalues of $A_t$ for all $t\in[T]\cup \{0\}$. For the first term in \eqref{eq:proofUCB1a}, by Lemma~\ref{lm:Xinshang_lemma} and the fact that $\lVert x_{(t,k)}\rVert_{A^{-1}_{t}}\leq\lVert x_{(t,k)}\rVert_{A^{-1}_{t,k-1}}$ for all $t\in[T]$ and $k\in[K]$, we obtain
\begin{align*}
\sum_{t=1}^T
\sum_{k=1}^K
\left(
\lVert x_{(t,k)}\rVert_{A^{-1}_{t-1}}
-
\lVert x_{(t,k)}\rVert_{A^{-1}_{t,k-1}}
\right)
& \leq
\sum_{t=1}^T
\sum_{i=1}^K
\left(
\lVert x_{(t,k)}\rVert_{A^{-1}_{t-1}}
-
\lVert x_{(t,k)}\rVert_{A^{-1}_{t}}
\right)\\
& \leq 
\sum_{t=1}^T
\sum_{i=1}^K
\left(
\sum_{j=1}^d \frac{2}{\sqrt{\lambda_{t-1,j}}}
-
\sum_{j=1}^d \frac{2}{\sqrt{\lambda_{t,j}}}
\right)
\\
&=
K\sum_{t=1}^T
\left(
\sum_{j=1}^d \frac{2}{\sqrt{\lambda_{t-1,j}}}
-
\sum_{j=1}^d \frac{2}{\sqrt{\lambda_{t,j}}}
\right)
\\
&=
K
\left(
\sum_{j=1}^d \frac{2}{\sqrt{\lambda_{0,j}}}
-
\sum_{j=1}^d \frac{2}{\sqrt{\lambda_{T,j}}}
\right)
\\
&\leq
K\sum_{j=1}^d \frac{2}{\sqrt{\lambda_{0,j}}}.
\end{align*}
Since $A_0=I_{d\times d}$, we have $\lambda_{0,j} = 1$ for all $j \in [d]$.
Hence, we have
\begin{align}\label{eq:UCB1_part2}
\begin{split}
& 2\alpha\sum_{t=1}^T
\sum_{i=1}^K
\left(
\lVert x_{(t,k)}\rVert_{A^{-1}_{t-1}}
-
\lVert x_{(t,k)}\rVert_{A^{-1}_{t,k-1}}
\right)\\
\leq & 2 \alpha \cdot K\sum_{j=1}^d \frac{2}{\sqrt{\lambda_{0,j}}}\\
= & 2 \alpha \cdot 2Kd\\
= & 2 \left(\sqrt{d\log{\left(\frac{1+TN}{\delta}\right)}}+1\right) \cdot 2Kd\\
= & 4d^{\frac{3}{2}}K\sqrt{\log\left(\frac{1+TN}{\delta}\right)} + 4Kd.
\end{split}
\end{align}
Combining (\ref{eq:UCB1_part1}) and (\ref{eq:UCB1_part2}), we complete the proof.
\halmos
\endproof

As shown by \citet{NIPS2011_4417}, $\UCBlin$ has a regret of $\tilde{O}(d \sqrt{T})$ in linear contextual bandit models, in which there is only one bandit observation in each period.
Hence, the variable cost of $\AdaptedOFUL$, $\tilde{O}(d \sqrt{KT})$, matches the same regret, since there are in total $KT$ bandit observations.

\subsection{Lower Bound on the Fixed Cost of $\AdaptedOFUL$}
\label{subsec:lower_bound}
In this section, we show a lower bound on the fixed cost of $\AdaptedOFUL$ by analyzing its regret in a simple example.

\begin{theorem}\label{thm:lb}
The regret of $\AdaptedOFUL$ is $\Omega(K \min(d,T))$.
\end{theorem}
\begin{proof}{Proof.}
Suppose there are $Kd$ products and they are split into $d$ groups, $(G_1,G_2,\cdots,G_d)$.
For each $i\in[d]$, $G_i$ has $K$ products with the same feature vector $\left(\frac{1}{2}+\frac{i}{2d}\right)e_i$, where $e_i$ denotes the unit vector with the $i$-th element being one.

Suppose $\theta^*=e_1$.
Then the optimal solution is to offer $G_1$ in all $T$ periods, and the expected reward in every period is $K\left(\frac{1}{2}+\frac{1}{2d}\right)$.

Initially, for each $i \in [d]$, the UCB value used by $\AdaptedOFUL$ for products in group $G_i$ is $\alpha \left(\frac{1}{2}+\frac{i}{2d}\right)$, which is increasing in $i \in [d]$. 

Since the feature vectors of products in different groups are mutually perpendicular, the UCB value of a product will not be affected by products selected from other groups. Therefore, the initial UCB value $\alpha \left(\frac{1}{2}+\frac{i}{2d}\right)$ for group $G_i$ will not change until one of the products in $G_i$ is picked by $\AdaptedOFUL$.

Given that $\AdaptedOFUL$ always picks products with the highest UCB values, and the initial UCB values increase in the group index $i \in [d]$, $\AdaptedOFUL$ will not select any product in $G_1$ in the first $d-1$ periods. 

Therefore, the total reward of $\AdaptedOFUL$ in the first $\min(T,d-1)$ periods must be zero. It follows that the regret of $\AdaptedOFUL$ is at least
\[ \min(T,d-1)  \cdot K\left(\frac{1}{2}+\frac{1}{2d}\right) = \Omega(K \min(T,d)).\]
\halmos

\end{proof}

We make two remarks regarding this lower bound result:
\begin{enumerate}
\item When the retailer selects $K$ products in the first period, she has no prior information for the estimation of $\theta^*$.
Thus, regardless of the value of $T$, the regret of any non-anticipating algorithm is at least $\Omega(K)$.
Compared to this, the lower bound $\Omega(K \min(T,d))$ for $\AdaptedOFUL$ has an extra factor of $\min(T,d)$. Therefore, in order to achieve a fixed cost close to $\Omega(K)$, we have to design a new algorithm, which is our goal in the next section.
\item When $T=d$, the lower bound on the regret for $\AdaptedOFUL$ is $\Omega(Kd)$, and there is only a $\tilde O(d^{1/2})$ gap compared to the fixed cost $\tilde O(Kd^{3/2})$ that we have proved in Theorem \ref{thm:UCB1}. This $\tilde O(d^{1/2})$ factor originates from the parameter $\alpha = \sqrt{d\log{\left(\frac{1+TN}{\delta}\right)}}+1$ that scales the lengths of confidence intervals.
\end{enumerate}

%% file: sec-04-ucb2.tex
In this section, we present a new algorithm $\UCBours$ for the semi-bandit model with linear generalization.
The new algorithm is based on a novel construction of exploration steps that are more \emph{conservative} than standard UCB procedures. 
Our analysis of $\UCBours$ shows its regret bound is $\tilde{O}(K\sqrt{d}+d\sqrt{KT})$, which improves on $\AdaptedOFUL$'s fixed cost by a factor of $d$.

\subsection{$\UCBours$ Algorithm for Semi-Bandits with Linear Generalization}

As illustrated in the lower bound proof in Section~\ref{subsec:lower_bound}, $\AdaptedOFUL$'s potential weakness lies in its tendency to select products with feature vectors that have the same or similar directions.
In other words, the set of selected products in each period is sometimes not diversified enough.
This is because, in any period $t$, the $K$ products in $S_t$ are selected independently based on the UCB values $p_t(i)$, which are calculated at the start of the period and sometimes form clusters for product groups with similar feature vector directions.

The new algorithm $\UCBours$ solves $\AdaptedOFUL$'s potential problem by offering more diversified product sets.
This is achieved through a sequential selection mechanism in each period, based on an adaptive product score $p_{t,k}(i)$, which shrinks per selection to make dissimilar products more likely to be chosen in subsequent selections.

More precisely, in $\UCBours$, the $K$ products in $S_t$ are selected sequentially, in a fashion that the $k$-th selection (for all $k\in[K]$) is based on scores $p_{t,k}(i)$ that are updated using the feature vectors of the $k-1$ previously selected products in the period.
If product $i$ is the $k$-th selection in period $t$, then the scores $p_{t,k+1}(j),\ldots,p_{t,K}(j)$ for all products $j$ that share similar feature vectors with product $i$ are decreased. This encourages more diversified product selections in each period.
Because of the way $p_{t,k}(i)$ is defined in $\UCBours$, it is always less than or equal to the standard UCB value $p_t(i)$.
This makes the new algorithm a variant of $\AdaptedOFUL$ with more conservative exploration steps -- hence the name $\UCBours$.

\
\\
$\UCBours$ algorithm for semi-bandits with linear generalization (with input parameter $\alpha$):
\begin{enumerate}
    \item Initialize $A_0= I_{d\times d}$ and $b_0=\vec{0}_{d}$.
    \item Repeat for $t=1,2,3,\ldots,T$
    \begin{itemize}
        \item[\ \ ] (a) Set $\theta_t= A_{t-1}^{-1}b_{t-1}$ and $A_{t,0}=A_{t-1}$.
        Initialize $S_t= \{\}$.
        \item[\ \ ] (b) Repeat for $k=1,2,3,\ldots,K$
        \begin{itemize}
            \item[\ \ \ \ ] i. Calculate $p_{t,k}(i)= x_{i}^\top\theta_t
          -\alpha\sqrt{x_{i}^\top A_{t-1}^{-1}x_{i}}
          +2\alpha\sqrt{x_{i}^\top A_{t,k-1}^{-1}x_{i}}$, for all $i\in[N]$.
            \item[\ \ \ \ ] ii. Add a product $j\in \argmax\limits_{i\in\cN\backslash S_t}\left\{p_{t,k}(i)\right\}$ to $S_t$.
            Update $A_{t,k}=A_{t,k-1}+x_{j}x_{j}^\top$.
        \end{itemize}
        \item[\ \ ] (c) Offer product set $S_t$.
        Observe outcomes $r_{t,i}\in\{0,1\}$ for $i\in S_t$.
        \item[\ \ ] (d) Update $A_t= A_{t-1}+\sum_{i\in S_t}x_{i}x_{i}^\top$ and $b_t= b_{t-1}+\sum_{i\in S_t}r_{t,i}x_{i}$.
    \end{itemize}
\end{enumerate}

\

$\UCBours$ shares a similar framework with $\AdaptedOFUL$, but it differs from $\AdaptedOFUL$ in Step~2(b), in which a different product score, $p_{t,k}(i)$, is maintained and the $K$ products in $S_t$ are chosen in a sequential manner.
We stress that the construction of product scores $p_{t,k}(i)$, i.e., Step 2(b)i., is novel.

\begin{figure}[h]
    \centering
    \includegraphics[width=0.95\textwidth]{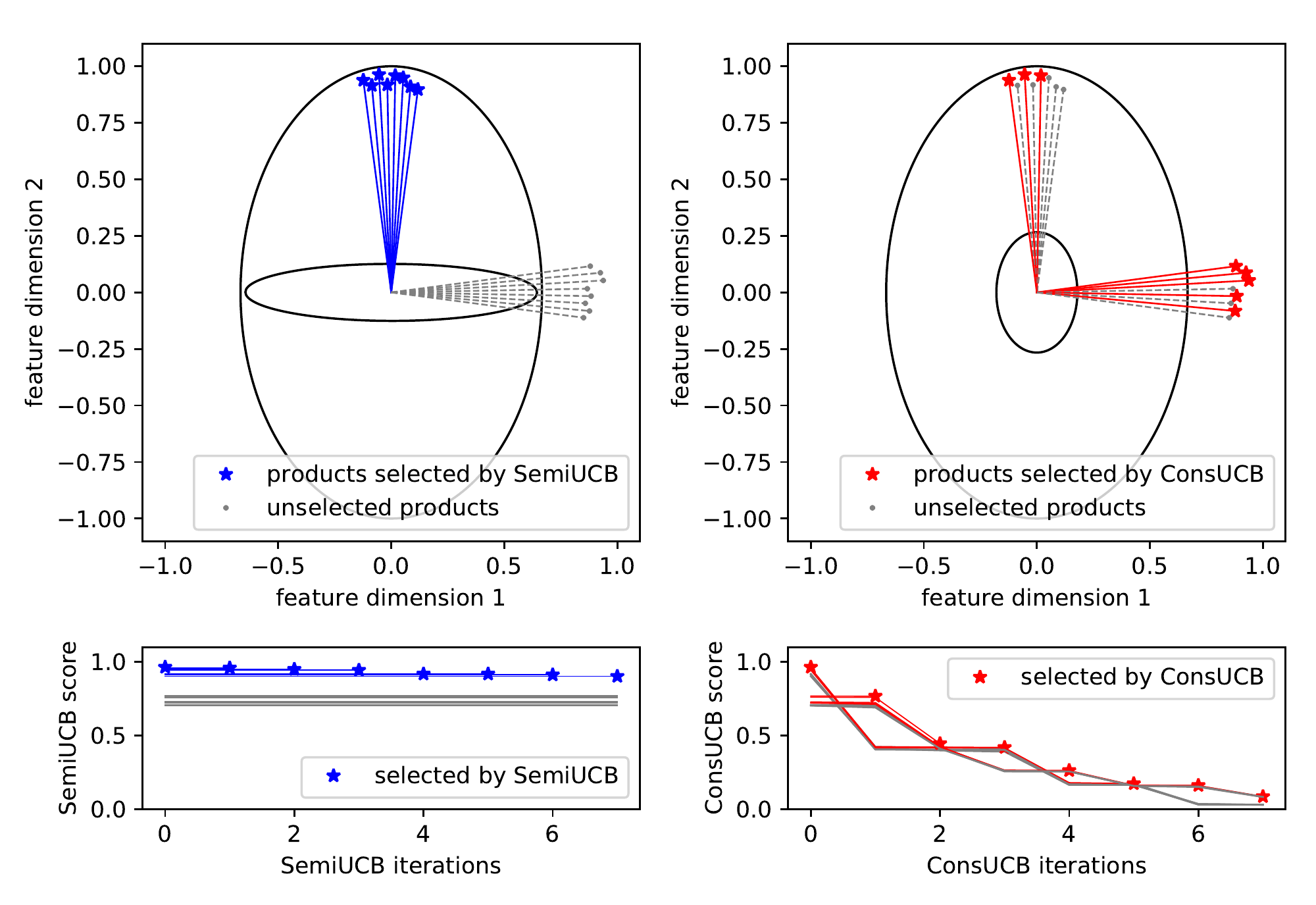}
    \caption{In the upper graphs, each dot represents a product in the two-dimensional feature space. We show that the products selected by $\UCBours$ are more diversified than those selected by $\AdaptedOFUL$ in a given time period.
    In the lower plots, we demonstrate that the scores of products used by $\UCBours$ shrink every time a product is selected, while the scores used by $\AdaptedOFUL$ are constant within each period.}
    \label{fig:compareUCB}
\end{figure}

Figure~\ref{fig:compareUCB} further illustrates the conservative exploration technique. In the figure, the sets of products selected by $\UCBours$ and $\AdaptedOFUL$ in a given time period are compared.
In this example, we have $d=2$, $K=8$, and $N=16$ products are clustered into two groups along the feature dimensions.
At the beginning of the period, suppose $A_{t-1} = \protect\begin{pmatrix} 1.5 & 0 \protect\\ 0 & 1\protect\end{pmatrix}$ and $\theta_t = \begin{bmatrix} 0 \\ 0\end{bmatrix}$ in both algorithms. This setting slightly favors exploration in feature dimension 2.
The upper graphs illustrate that $\AdaptedOFUL$ selects $K$ products all in one group, reducing the uncertainty related to $\theta^*$ almost only in one dimension (shown by the shrinkages from the larger ellipses to the smaller ones), whereas $\UCBours$ selects a more diversified product set, significantly reducing $\theta^*$ uncertainty in both dimensions.
The lower two graphs demonstrate that $\AdaptedOFUL$ calculates the product scores $p_t(\cdot)$ only once at the beginning and chooses the top $K$ products, while the scores $p_{t,k}(\cdot)$ in $\UCBours$ are updated every time a product is selected.

\subsection{Regret Bound Analysis for $\UCBours$}

In this section, we prove that the regret of $\UCBours$ is $\tilde{O}(K\sqrt{d}+d\sqrt{KT})$.

Proposition~\ref{pp:LCB} demonstrates the key benefit of using conservative exploration. It shows that, under $\UCBours$, the regret per product selection can be upper-bounded by the sum of a standard confidence interval term and the increase in the \emph{lower confidence bound} of a product in the optimal product set.
The rest of the regret analysis follows from Proposition~\ref{pp:LCB} and is completed in Theorem~\ref{thm:UCB2}.

For convenience, let $\Delta_{t,k}(i)$ denote $\sqrt{x_{i}^\top A_{t,k}^{-1}x_{i}}$, for all $i\in[N]$, $k\in[K]$ and $t\in[T+1]$.
For each product $i\in[N]$ and each period $t\in[T+1]$, define $\texttt{LCB}_t(i)=x_{i}^\top\theta_t-\alpha\Delta_{t,0}(i)$ and $\texttt{UCB}_t(i)=x_{i}^\top\theta_t+\alpha\Delta_{t,0}(i)$ as the lower and upper confidence bounds for $\mu(i)=x_{i}^\top\theta^*$, respectively.
Let $i_{t,k} \in [N]$ denote the $k$-th product selected in period $t$ by $\UCBours$.

\begin{proposition}
\label{pp:LCB}
If $\UCBours$ is run with $\alpha=\sqrt{d\ln{\left(\frac{1+N+TN}{\delta}\right)}}+1$, then with probability at least $1-\delta$, the following conditions hold for all $t\in[T]$ and $k\in[K]$:
\begin{itemize}
	\item[(1)] If $i_{t,k}\notin S^*$, then for all $i^*\in S^*\backslash S_t$,
	\[
	\mu(i^*)-\mu(i_{t,k})\leq\texttt{LCB}_{t+1}(i^*)-\texttt{LCB}_t(i^*)+2\alpha\Delta_{t,k-1}(i_{t,k}).
	\]
	\item[(2)] If $i_{t,k}\in S^*$, then
	\[
	0\leq\texttt{LCB}_{t+1}(i_{t,k})-\texttt{LCB}_{t}(i_{t,k})+2\alpha\Delta_{t,k-1}(i_{t,k}).
	\]
\end{itemize}
\end{proposition}
\proof{Proof.}
By Lemma~\ref{lm:qin_lem_4.1}, with probability at least $1-\delta$, we have
\begin{equation}
\label{eq:LCBUCB}
\texttt{LCB}_{t}(i)= x_{i}^\top\theta_t-\alpha\Delta_{t,0}(i)\leq \mu(i) \leq x_{i}^\top\theta_t + \alpha\Delta_{t,0}(i)=\texttt{UCB}_{t}(i),
\end{equation}
for all $i\in[N]$ and  $t\in[T+1]$.

Consider any product $i_{t,k}$, which is selected in the $k$-th step in period $t$ by $\UCBours$.
Conditioned on (\ref{eq:LCBUCB}), we want to show, for all $i_{t,k}\in S_t$,
\begin{enumerate}
	\item If $i_{t,k}\notin S^*$, then for all $i^*\in S^*\backslash S_t$,
	\begin{equation}
	\label{eq:notinoptimalset}
	\mu(i^*)-\mu(i_{t,k})\leq\texttt{LCB}_{t+1}(i^*)-\texttt{LCB}_{t}(i^*)+2\alpha\Delta_{t,k-1}(i_{t,k}).
	\end{equation}
	\item If $i_{t,k}\in S^*$, then
	\begin{equation}
	\label{eq:inoptimalset}
	0\leq\texttt{LCB}_{t+1}(i_{t,k})-\texttt{LCB}_{t}(i_{t,k})+2\alpha\Delta_{t,k-1}(i_{t,k}).
	\end{equation}
\end{enumerate}

Consider the first case where $i_{t,k}\notin S^*$. We have
\begin{align*}
\mu(i^*)-\mu(i_{t,k})
&\overset\da\leq 
\texttt{UCB}_{t+1}(i^*)-\texttt{LCB}_{t}(i_{t,k})
\\
&\overset\db=\texttt{UCB}_{t+1}(i^*)-p_{t,k}(i_{t,k})+2\alpha\Delta_{t,k-1}(i_{t,k})
\\
&\overset\dc\leq\texttt{UCB}_{t+1}(i^*)-p_{t,k}(i^*)+2\alpha\Delta_{t,k-1}(i_{t,k})
\\
&\overset\dd=\left(\texttt{LCB}_{t+1}(i^*)+2\alpha\Delta_{t+1,0}(i^*)\right)-\left(\texttt{LCB}_{t}(i^*)+2\alpha\Delta_{t,k-1}(i^*)\right)+2\alpha\Delta_{t,k-1}(i_{t,k})
\\
&\overset\de\leq\texttt{LCB}_{t+1}(i^*)-\texttt{LCB}_{t}(i^*)+2\alpha\Delta_{t,k-1}(i_{t,k}).
\end{align*}
Above, inequality $\da$ follows from condition \eqref{eq:LCBUCB};
equalities $\db$ and $\dd$ are by definition of $p_{t,k}(i)$, $\texttt{UCB}_{t}(i)$ and $\texttt{LCB}_{t}(i)$; inequality $\dc$ is because product $i_{t,k}$ is selected, instead of $i^*$, by $\UCBours$ in the $k$-th step in period $t$; inequality $\de$ is because
\[
\Delta_{t+1,0}(i)=\sqrt{x_{}^\top A_{t}^{-1}x_{i}}=\sqrt{x_{}^\top A_{t,K}^{-1}x_{i}}
\leq
\sqrt{x_{i}^\top A_{t,k-1}^{-1}x_{i}}
=
\Delta_{t,k-1}(i)
\quad\mbox{for all }i\in[N].
\]

Now consider the second case where $i_{t,k}\in S^*$.

Given condition \eqref{eq:LCBUCB} and $\Delta_{t+1,0}(i_{t,k})\leq\Delta_{t,k-1}(i_{t,k})$, we have
\[
\begin{aligned}
0\leq\texttt{UCB}_{t+1}(i_{t,k})-\texttt{LCB}_{t}(i_{t,k})&=\texttt{LCB}_{t+1}(i_{t,k})-\texttt{LCB}_{t}(i_{t,k})+2\alpha\Delta_{t+1,0}(i_{t,k})
\\
&\leq\texttt{LCB}_{t+1}(i_{t,k})-\texttt{LCB}_{t}(i_{t,k})+2\alpha\Delta_{t,k-1}(i_{t,k}).
\end{aligned}
\]

Since (\ref{eq:LCBUCB}) holds with probability at least $1-\delta$, the probability that either (\ref{eq:notinoptimalset}) or (\ref{eq:inoptimalset}) holds is also at least $1-\delta$.
\halmos
\endproof

Now we complete the regret analysis for $\UCBours$, using the results from Proposition~\ref{pp:LCB}.
\begin{theorem}\label{thm:UCB2}
If $\UCBours$ is run with $\alpha=\sqrt{d\ln{\left(\frac{1+N+TN}{\delta}\right)}}+1$, then with probability at least $1-\delta$, the regret of the algorithm is $\tilde O(K\sqrt{d} + d \sqrt{KT})$.
\end{theorem}

\begin{proof}{Proof.}
Notice that the regret in period $t$ can written as
\begin{align*}
\sum_{i\in S^*}\mu(i)
-
\sum_{i\in S_t}\mu(i)
&
=
\left(
\sum_{i\in S^*\cap S_t}\mu(i)
+
\sum_{i\in S^*\backslash S_t}\mu(i)
\right)
-
\left(
\sum_{i\in S_t\cap S^*}\mu(i)
+
\sum_{i\in S_t\backslash S^*}\mu(i)
\right)
\\
&=
\sum_{i\in S^*\backslash S_t}\mu(i)
-
\sum_{i\in S_t\backslash S^*}\mu(i).
\end{align*}

Since $S_t\backslash S^*$ and $S^*\backslash S_t$ have the same number of products, let $f_t:S_t\backslash S^*\to S^*\backslash S_t$ be an arbitrary one-to-one function that maps from $S_t\backslash S^*$ to $S^*\backslash S_t$.
As a result, summing over $S^*\backslash S_t$ is the same as summing over $\left\{f_t(i):i \in S_t\backslash S^*\right\}$.
Hence, the $T$-period total expected regret can be written as
\begin{align*}
R(T)&=\sum_{t=1}^T\left(\sum_{i\in S^*\backslash S_t}\mu(i)
-
\sum_{i\in S_t\backslash S^*}\mu(i)\right)\\
&=\sum_{t=1}^T\sum_{i \in S_t\backslash S^*}\left(\mu(f_t(i))-\mu(i)\right)\\
& = \sum_{t=1}^T \sum_{k=1}^K \bI_{\{i_{t,k} \not\in S^*\}} \left(\mu(f_t(i_{t,k}))-\mu(i_{t,k})\right),
\end{align*}
where $\bI_{\{\cdot\}}$ is the indicator function.

Then, we use Proposition~\ref{pp:LCB} to obtain
\begin{align*}
R(T)&= \sum_{t=1}^T \sum_{k=1}^K \bI_{\{i_{t,k} \not\in S^*\}} \left(\mu(f_t(i_{t,k}))-\mu(i_{t,k})\right)\\
& = \sum_{t=1}^T \sum_{k=1}^K \left[ \bI_{\{i_{t,k} \not\in S^*\}}\cdot \left(\mu(f_t(i_{t,k}))-\mu(i_{t,k})\right)+ \bI_{\{i_{t,k} \in S^*\}} \cdot 0 \right]\\
&\leq \sum_{t=1}^T \sum_{k=1}^K \left[ \bI_{\{i_{t,k} \not\in S^*\}}\cdot \left(\texttt{LCB}_{t+1}(f_t(i_{t,k}))-\texttt{LCB}_{t}(f_t(i_{t,k}))+2\alpha\Delta_{t,k-1}(i_{t,k}) \right) \right. \\
 & \quad \quad \quad \quad + \left. \bI_{\{i_{t,k} \in S^*\}} \cdot  \left( \texttt{LCB}_{t+1}(i_{t,k})-\texttt{LCB}_{t}(i_{t,k})+2\alpha\Delta_{t,k-1}(i_{t,k}) \right) \right]\\
&= \sum_{t=1}^T \left[ \sum_{i \in S^*} \left( \texttt{LCB}_{t+1}(i)-\texttt{LCB}_{t}(i) \right) + \sum_{k=1}^K 2\alpha\Delta_{t,k-1}(i_{t,k}) \right]\\
& =  \sum_{i \in S^*} \left( \texttt{LCB}_{T+1}(i)-\texttt{LCB}_{1}(i) \right) + \sum_{t=1}^T \sum_{k=1}^K 2\alpha\Delta_{t,k-1}(i_{t,k}).
\end{align*}

Conditioned on \eqref{eq:LCBUCB}, we have $\texttt{LCB}_{T+1}(i) \leq \mu(i)$ for all $i \in [N]$. Since we assume $\mu(i) \leq 1$ for all products $i \in [N]$, we have $\texttt{LCB}_{T+1}(i) \leq 1$ for all $i \in S^*$. Thus, 
\[ \sum_{i \in S^*} \texttt{LCB}_{T+1}(i) \leq |S^*| = K.\]

Moreover, by definition of $\texttt{LCB}_{1}(i)$, we have
\[ \texttt{LCB}_{1}(i) = - \alpha \sqrt{x_i^\top A_0^{-1} x_i} = -\alpha \sqrt{x_i^\top I_{d\times d} x_i} \geq -\alpha.\]
Thus,
\[ \sum_{i \in S^*} -\texttt{LCB}_{1}(i) \leq \alpha |S^*| = \alpha K. \]

Finally, by Lemma~\ref{lm:chu_lem_3}, we have
\[
2\alpha\sum_{t=1}^T\sum_{k=1}^K\Delta_{t,k-1}(i_{t,k})\leq10\alpha\sqrt{dKT\log(KT)}.
\]

Altogether, the total regret of $\UCBours$ can be bounded by
\begin{align*}
R(T)& \leq \sum_{i \in S^*} \left( \texttt{LCB}_{T+1}(i)-\texttt{LCB}_{1}(i) \right) + \sum_{t=1}^T \sum_{k=1}^K 2\alpha\Delta_{t,k-1}(i_{t,k})\\
& \leq K + \alpha K + 10\alpha\sqrt{dKT\log(KT)}\\
& = K + \left(\sqrt{d\ln{\left(\frac{1+N+TN}{\delta}\right)}}+1\right) \left(K + 10 \sqrt{dKT\log(KT)}\right)\\
& = \tilde O(K \sqrt{d} + d\sqrt{KT}).
\end{align*}
\halmos

\end{proof}

%% file: sec-05-numerical.tex
In this section, we test the algorithms $\AdaptedOFUL$ and $\UCBours$ using industrial data from Alibaba Group.

%

We select around $N \approx 20,000$ different products in the toys category that were sold on Alibaba Tmall, the largest business-to-consumer online retail platform in China, during the month of April 2018.
Since our model is only concerned with products with medium to low demands, the set of $N \approx 20,000$ products excludes any products that on average sold more than 20 units per day in China during that month.

The feature vectors of the products are generated by Alibaba's deep learning team using representation learning \citep{Bengio2013}, based on both product's intrinsic features as well as historical product-related user activities.
Thus, each feature vector is essentially compressed from high-dimensional sparse data to a $d = 50$ dimensional vector.
As we have shown in Section~\ref{sec:intro}, there is a near-linear relationship between the products' features and their probabilities of being purchased.

For an arbitrarily chosen city in China, we estimate the parameter $\theta^*$ of a linear predictor based on the city's local sales data in April 2018.
Then, assuming $\theta^*$  is unknown at the beginning, we test the performances of $\AdaptedOFUL$ and $\UCBours$ on the product selection problem, by computing their cumulative regrets, the algorithms' compromises in performance compared to the case where $\theta^*$ is known from the beginning.

\begin{figure}[h]
    \centering
    \includegraphics[width=0.53\textwidth]{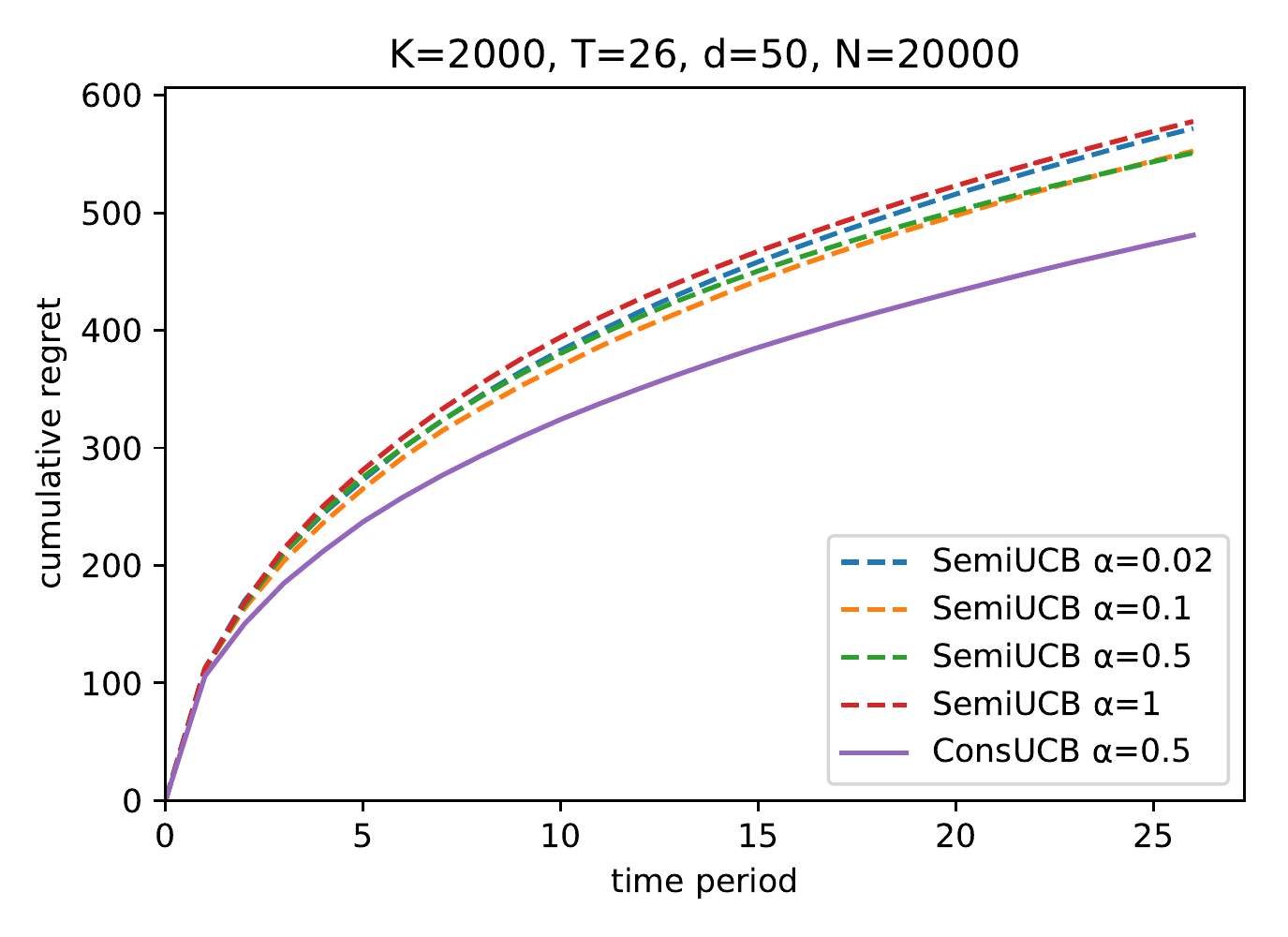}
    \caption{Cumulative regret of $\AdaptedOFUL$ (with $\omega=1$) and $\UCBours$ in the first $T=26$ periods, when $K=2000$.}
    \label{fig:regretTrend1}
\end{figure}

Figures \ref{fig:regretTrend1} and \ref{fig:regretTrend23} summarize our simulation results, for $K=2000$, $K=1000$ and $K=200$ respectively.
For each $K\in\{2000,1000,200\}$, we run both $\AdaptedOFUL$ and $\UCBours$ for 26 time periods, and in each period $t\in\{1,\ldots,26\}$, we calculate the algorithms' per-period regrets, $\texttt{regret}_t(\AdaptedOFUL)$ and $\texttt{regret}_t(\UCBours)$, and then compute the algorithms' cumulative regrets, $\sum_{s=1}^{t}\texttt{regret}_s(\AdaptedOFUL)$ and $\sum_{s=1}^{t}\texttt{regret}_s(\UCBours)$.
Averaged over ten simulation replicates, the two algorithms' cumulative regrets are plotted in Figures~\ref{fig:regretTrend1} and \ref{fig:regretTrend23}. In Table~\ref{table:regrets}, we list $\UCBours$'s improvements in cumulative regret over $\AdaptedOFUL$ at the end of the 26-period time horizon.



\begin{figure}[h]
    \centering
\begin{subfigure}{.5\textwidth}
  \centering
  \includegraphics[width=0.98\linewidth]{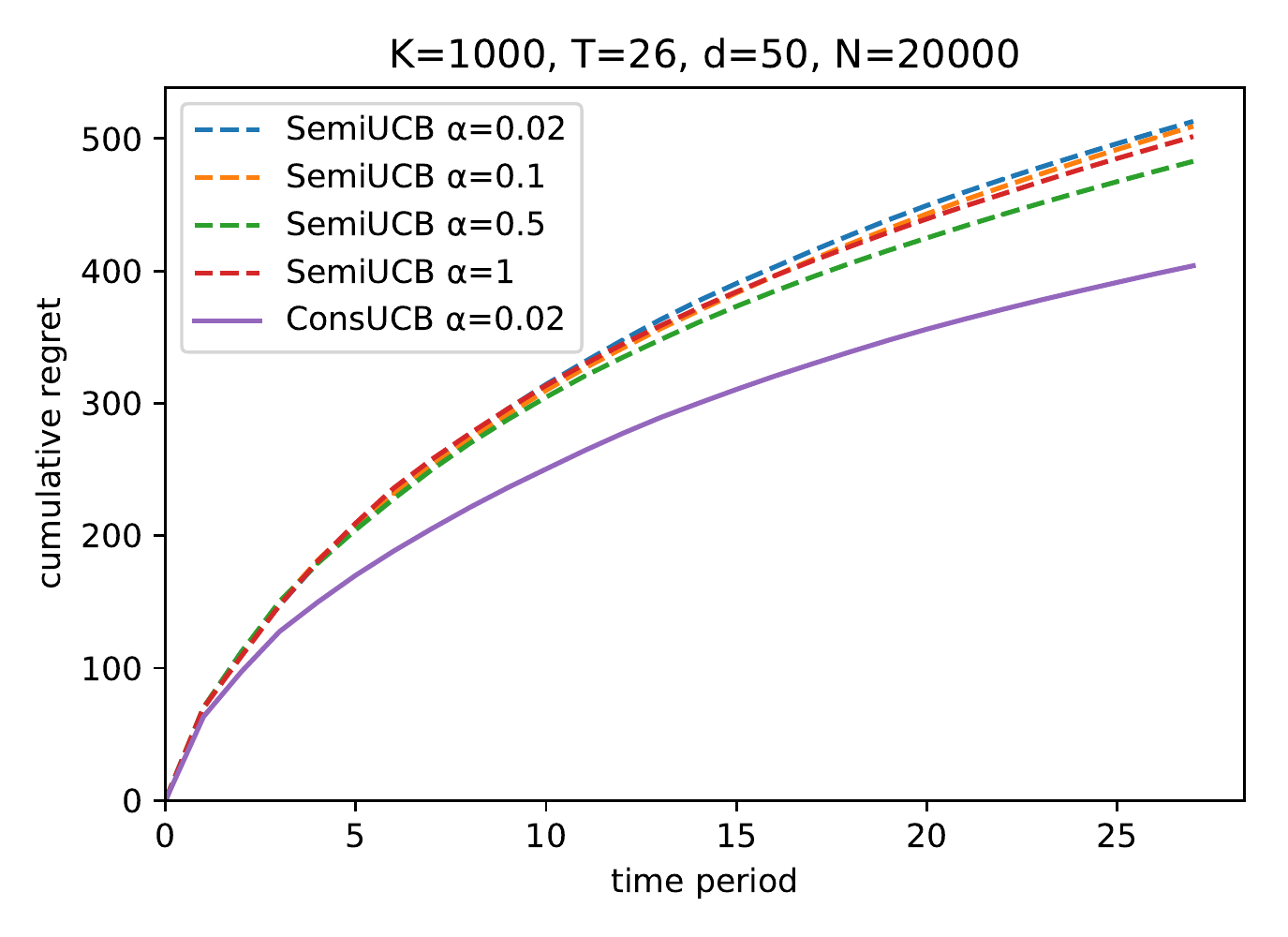}
  \label{fig:sub1}
\end{subfigure}%
\begin{subfigure}{.5\textwidth}
  \centering
  \includegraphics[width=0.98\linewidth]{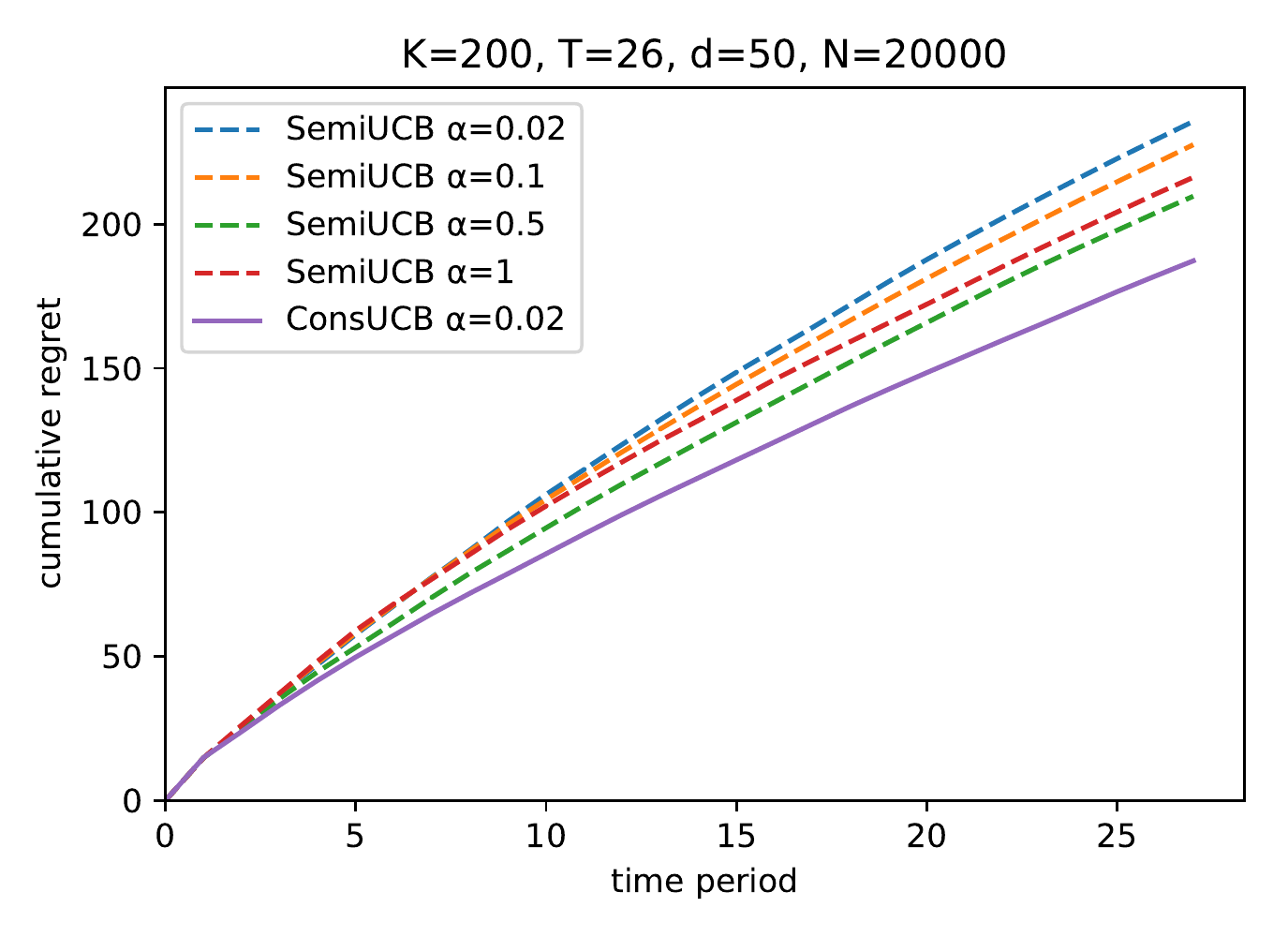}
  \label{fig:sub2}
\end{subfigure}
    \caption{Cumulative regret of $\AdaptedOFUL$ (with $\omega=1$) and $\UCBours$ in the first $T=26$ periods, when $K=1000$ and $K=200$.}
    \label{fig:regretTrend23}
\end{figure}

It is clear that $\UCBours$'s cumulative regret is consistently lower than that of $\AdaptedOFUL$ over a set of $\alpha$ values, for all values of $K$ that we include in the test.
Moreover, from all plots, we can see that, compared to $\AdaptedOFUL$, $\UCBours$'s most reductions in regret happen in the first ten periods, and their differences stabilize starting around the \nth{15} period.
This is mainly because $\UCBours$ is specially designed to mitigate the fixed cost part of the regret, which is introduced right at the beginning, independent of $T$.
While $T$ increases as the algorithms progress, the estimation of $\theta^*$ in both $\AdaptedOFUL$ and $\UCBours$ becomes more accurate, and the variable cost part of the regret takes over.
Since both algorithms have the same order of regret for the variable cost, their differences eventually become stable.

\begin{table}[h!]
\centering
\begin{tabular}{|c|c|c|c|c|c|}
\hline
 & \multicolumn{2}{c|}{$\AdaptedOFUL$} & \multicolumn{2}{c|}{$\UCBours$} & Regret improvement \\
\hline
 \multirow{4}{*}{$K=2000$} & $\alpha=0.02$ & 571.80 & \multirow{4}{*}{$\alpha=0.50$} & \multirow{4}{*}{480.96} & 15.89\%  \\
 & $\alpha=0.10$ & 552.56 & & & 12.96\%\\
 & $\alpha=0.50$ & 551.01 & & & \textbf{12.71\%}\\
 & $\alpha=1.00$ & 577.77 & & & 16.76\%\\
\hline
 \multirow{4}{*}{$K=1000$} & $\alpha=0.02$ & 512.91 & \multirow{4}{*}{$\alpha=0.02$} & \multirow{4}{*}{403.88} & 21.26\%  \\
 & $\alpha=0.10$ & 509.25 & & & 20.69\%\\
 & $\alpha=0.50$ & 482.78 & & & \textbf{16.34\%}\\
 & $\alpha=1.00$ & 501.52 & & & 19.47\%\\
\hline
 \multirow{4}{*}{$K=200$} & $\alpha=0.02$ & 235.6 & \multirow{4}{*}{$\alpha=0.02$} & \multirow{4}{*}{187.28} & 20.53\%  \\
 & $\alpha=0.10$ & 227.58 & & & 17.70\%\\
 & $\alpha=0.50$ & 209.69 & & & \textbf{10.69\%}\\
 & $\alpha=1.00$ & 216.25 & & & 13.39\%\\
\hline
\end{tabular}
\caption{Table of cumulative regrets of $\AdaptedOFUL$ and $\UCBours$ at the end of the 26-period time horizon and $\UCBours$'s relative improvements over $\AdaptedOFUL$ in cumulative regrets.}
\label{table:regrets}
\end{table}

Table~\ref{table:regrets} shows that $\UCBours$ improves $\AdaptedOFUL$'s best performance over a set of $\alpha$ values by 12.71\%, 16.34\% and 10.69\%, for $K=2000$, $K=1000$ and $K=200$ respectively.
It is interesting to note that the best improvement comes from the experiment with $K=1000$, rather than the one with the largest value $K=2000$.
One main reason is that, for a fixed $N$ number of products, if $K$ is increased above a certain threshold, then $\AdaptedOFUL$ is forced to select a diversified product set, and thus closes its gap with $\UCBours$.
Indeed, in the most extreme cases, where $K=N$ or $K=1$, $\AdaptedOFUL$ and $\UCBours$ perform exactly the same.
However, when $K$ is much smaller than $N$, we expect $\UCBours$'s improvement over $\AdaptedOFUL$ to grow in a monotone fashion as $K$ increases.

The problem instance in this numerical study is clearly not the worst case for $\AdaptedOFUL$, as the algorithm's regret per period is much smaller than the worst bound $\Omega(K)$ proved in Theorem \ref{thm:lb}.
However, since the feature vectors are generated by a deep learning method, which is almost a black box to the online algorithms, there is no guarantee that, if the problem instance is repeated in other sessions or in other cities, the set of feature vectors is always in $\AdaptedOFUL$'s favor to avoid the worst-case type of regret of $\Omega(Kd)$.

Regarding the practicality of the online learning algorithm solutions, we measure the number of SKU replacements in the offered product sets over the time periods.
A high number means a large portion of the product sets need to be replaced in a period, while a low number suggests the product set only needs minimal adjustments.
The plots in Figure~\ref{fig:change} summarize the numbers of SKUs replaced between subsequent periods for both $\AdaptedOFUL$ and $\UCBours$.
For all cases of $K\in[2000,1000,200]$, the number of per-period replacements for both algorithms drops quickly in the first ten periods, and it converges to around $10\%$ of the size of product set near the end of the first 50 periods.
Hence, we numerically show that both algorithms $\AdaptedOFUL$ and $\UCBours$ do not replace large portions of the product sets indefinitely often.
Moreover, the simulation results further suggest that, compared to $\AdaptedOFUL$ (with the best $\alpha$ parameter), $\UCBours$ reduces the total number of SKU replacements in the first 50 periods by 4.86\%, 10.73\% and 16.84\%, for $K=2000,1000$ and $200$ respectively.

\begin{figure}[h]
    \centering
\begin{subfigure}{.33\textwidth}
  \centering
  \includegraphics[width=0.99\linewidth]{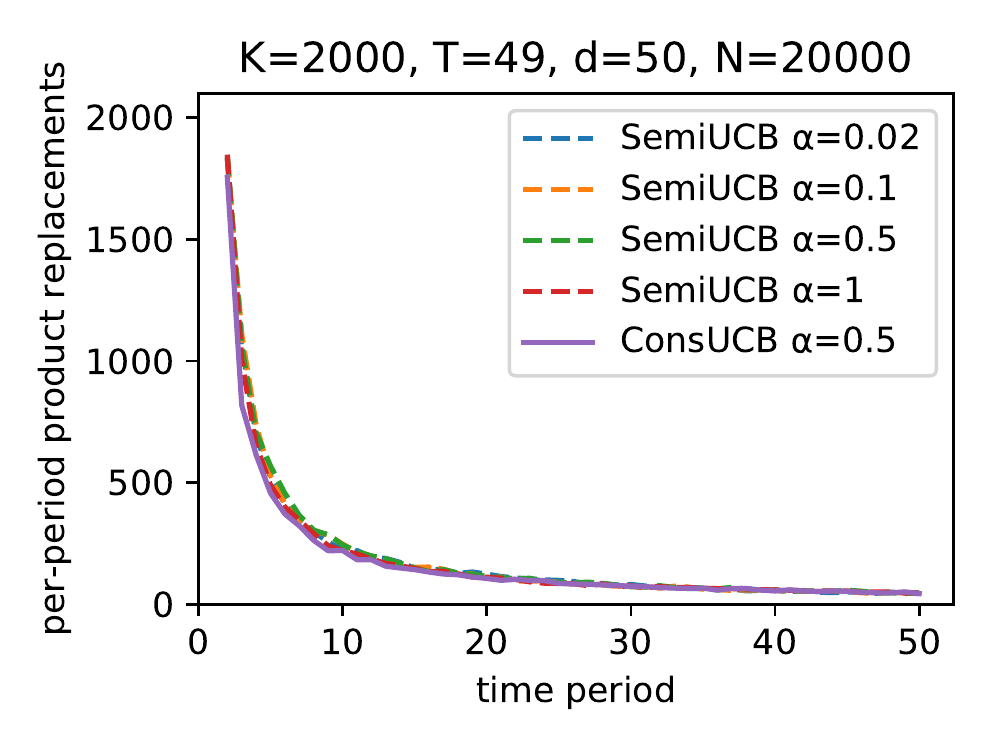}
  \label{fig:sub1}
\end{subfigure}%
\begin{subfigure}{.33\textwidth}
  \centering
  \includegraphics[width=0.99\linewidth]{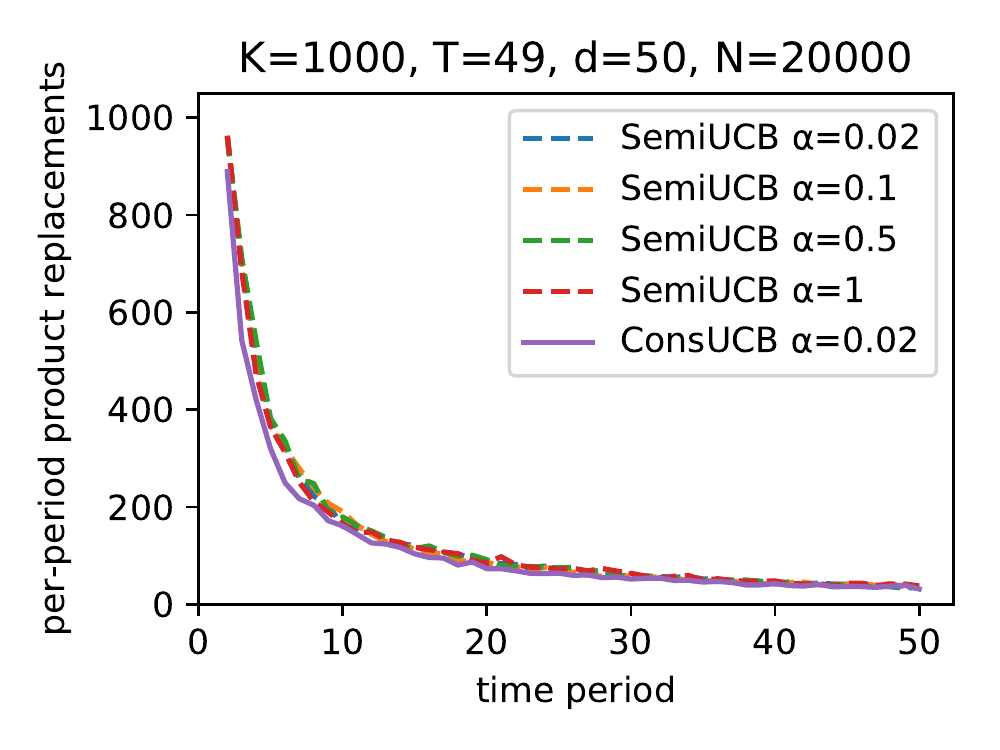}
  \label{fig:sub2}
\end{subfigure}%
\begin{subfigure}{.33\textwidth}
  \centering
  \includegraphics[width=0.99\linewidth]{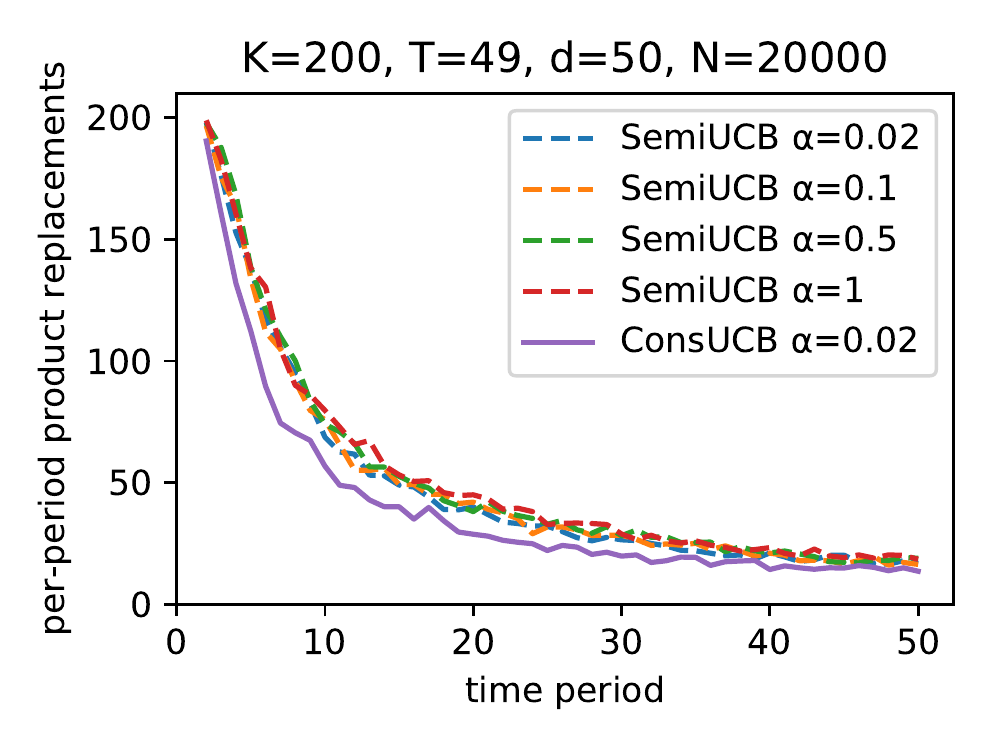}
  \label{fig:sub3}
\end{subfigure}
    \caption{Per-period number of SKU replacements in the offered product sets by $\AdaptedOFUL$ (with $\omega=1$) and $\UCBours$ in the first $T=50$ periods, when $K=2000$, $K=1000$ and $K=200$.}
    \label{fig:change}
\end{figure}

Therefore, given the simulation results on the improvement of cumulative regret and reduction of product set replacements, $\UCBours$ is the better option than $\AdaptedOFUL$, both theoretically and experimentally, for this e-commerce problem setting.

%% file: sec-06-conclusion.tex
In this paper, we study a product selection problem inspired by the crucial use of urban warehouses by online retailers for ultra-fast delivery services.
We formulate the problem in a semi-bandit model with linear generalization in product features.
Our alternative analysis of an existing standard UCB algorithm suggests the regret can be interpreted as the sum of two parts, i.e.\ a $T$-dependent ``variable cost'' part and a $T$-independent ``fixed cost'' part.
We propose a novel online learning algorithm called $\UCBours$, and show it improves the fixed-cost part of the regret, while keeping the variable cost in the same order.
In the specific model setting of the product selection problem in this paper, where $K$ is much larger than $T$ or $d$, the improvement in regret is even more significant, both in theory and in a numerical study on an online retailer's data.

One extension of the model is to incorporate the actual number of sales of each product, although the model in this paper is concerned only with products' probabilities of positive sales.
In many e-commerce businesses, multiple days of on-hand safety stock are always maintained in the inventory, regardless of inventory holding costs, to keep stock-out probabilities at a very low level, for the sake of customer experience.
As a result, retailers can most often directly observe the exact level of customer demand for each product offered on the fast-delivery platform (in other words, the observation of demand is not censored due to the limited inventory).
To incorporate the number of sales into our model, we can re-define $x_i^T \theta^*$ as the expected number of sales of product $i$, and replace the 0-1 bandit feedback with a random non-negative demand for the product.
The corresponding confidence bounds can be adapted following standard UCB techniques in the work by \citet{NIPS2011_4417}.

In the problem setting, we implicitly assume the positive-sales probabilities are the same across all periods.
Nevertheless, we can easily extend the model to capture demand changes due to day-of-the-week (or week-of-the-month) effects in the problem.
For example, in each period $t$, the probability of positive-sales for any product $i$ can be assumed to be $\gamma_t x_i^{\top} \theta^*$, where $\gamma_t$ can take seven different values scaling the demand for each day of the week.
As the scaling factors $\gamma_t$ can be easily estimated from aggregate sales data, the regret bound of any algorithm can be readily updated by rescaling its regret terms in all periods using the same factors.

For the purpose of presenting a clean model, we leave the previous potential extensions out of the model, since the main insights and techniques in our algorithms and analyses remain the same.
However, in reality, the challenges that urban-warehouse retailers face in optimizing their products/inventory are not merely selecting the top products; it is more complicated than the model even with the above extensions.
There may be shipping capacities into and out of urban warehouses.
There may be costs associated with retrieving unsold products from urban warehouses.
There may be city regions that can be simultaneously supplied by multiple urban warehouses, each with different product offer sets.

In the model, we assume the retailers have the ability to freely adjust the offered product set in each period according to the online learning algorithm.
It is clear that our model does not rule out the possibility of regularly replacing the entire offered product set in an urban warehouse on a weekly basis.
This is not only financially unviable, but practically irrational for retailers to do so on a regular basis, considering the shipping cost and other related costs.
Hence, in Section~\ref{sec:numerical}, we numerically show that our online learning algorithms limit the total number of products replaced out of urban warehouses within a reasonable range, compared to the total number of the products offered.
This further indicates that the costs related to product replacement is contained and converges as the learning becomes more accurate.
In simpler words, when using our online learning algorithms, replacing the a large portion of the offered product set on a regular basis is unlikely; even when it happens, it tends to happen in the early periods, which is more justifiable given the business objective on optimizing customer experience and user growth.

Moreover, it is also reasonable for one to suggest other business objectives in the model, for example, to maximize revenue.
This is indeed a different objective than that we define in the model, because additional factors like pricing may need to be considered.
Specifically, the objective defined in this paper is to identify products with top probabilities of meeting customer's demands, rather than products with top expected revenues, among the less popular products.
Nonetheless, we point out that, as the most popular products capture most of the revenue, the less popular products are offered to increase the retail platforms' product breadth and improve customer experience and retention.
In this case, revenue optimization may not be the top priority in our product selection problem, and, after all, revenue is just one of the performance indicators, that retailers consider in making business decisions, among market share, number of users, profit, and others.

Indeed, this research proposes and analyzes only an abstraction of the actual product/inventory optimization problem.
Nonetheless, the insights provided by the $\UCBours$ algorithm go beyond product selection:
(a) from a high level, we can view $\UCBours$ as an algorithm for generating a \emph{ranking} of products, as products are chosen by the algorithm sequentially without replacement in each period.
Analyses in this paper show that, compared to the standard UCB algorithm, $\UCBours$ encourages exploration with a guaranteed smaller fixed-cost regret, potentially outperforming other general machine learning algorithms producing rankings of products as well.
(b) In many inventory problem solvers, the demand forecast of every product is given as part of the input.
With $\UCBours$, we can use its conservative score value $p_{t,k}(\cdot)$ as a proxy for demand forecast.
This is another quick way of directly applying our conservative exploration technique to real business decisions.

%% file: sec-08-appendixUCB1.tex
Before proving Lemma~\ref{lm:Xinshang_lemma}, we first prove the following results:

\begin{lemma} \label{lm:key0}
Let $u \in \bR^d$ be any vector.
Let $x \in \bR^d$ be any unit vector such that $\|x\|_2 = 1$.
For any positive definite diagonal matrix $\Lambda \in \bR^{d\times d}$, we have
\[
\frac{(x^\top \Lambda u)^2}{\sqrt{x^\top \Lambda x}} \leq u^\top \Lambda^{1.5} u.
\]
\end{lemma}

\proof{Proof.}
We have
\[
u^\top \Lambda^{1.5} u - \frac{(x^\top \Lambda u)^2}{\sqrt{x^\top \Lambda x}} = u^\top \left( \Lambda^{1.5} - \frac{ \Lambda x x^\top \Lambda}{\sqrt{x^\top \Lambda x}}\right) u = u^\top \Lambda^{0.75} \left( I - \frac{ \Lambda^{0.25} x x^\top \Lambda^{0.25}}{\sqrt{x^\top \Lambda x}} \right) \Lambda^{0.75} u.
\]

Thus, it suffices to verify that $I - \frac{ \Lambda^{0.25} x x^\top \Lambda^{0.25}}{\sqrt{x^\top \Lambda x}}$ is positive semi-definite.
This is equivalent to showing all of its principal minors are nonnegative.

For any subset $J$ of the row (column) index set $[d]$, let $\Lambda_{J,J}$ be a principal submatrix of any matrix $\Lambda$ with rows and columns whose indices are in $J$.
Similarly, let $x_{j}$ be a subvector of any vector $x$ with elements whose indices are in $J$.
The principal minor of $I- \frac{ \Lambda^{0.25} x x^\top \Lambda^{0.25}}{\sqrt{x^\top \Lambda x}}$ with respect to index set $J$ is
\[
\text{det}\left(I_{J,J}- \frac{ \Lambda^{0.25}_{J,J} x_{J} x_{J}^\top \Lambda^{0.25}_{J,J}}{\sqrt{x^\top \Lambda x}}\right)
=
1-\frac{ x_{J}^\top \Lambda^{0.5}_{J,J} x_{J}}{\sqrt{x^\top \Lambda x}}.
\]
Since $\Lambda$ is a diagonal matrix,
\[
x_{J}^\top \Lambda^{0.5}_{J,J} x_{J}=\sum_{i\in J}\sqrt{\Lambda_{i,i}}x_i^2\leq\sum_{i=1}^d\sqrt{\Lambda_{i,i}}x_i^2=x^\top\Lambda^{0.5} x
\quad
\text{and}\quad
\sqrt{x^\top \Lambda x}=\sqrt{\sum_{i=1}^d\Lambda_{i,i}x_i^2}=\|\Lambda^{0.5}x\|_2.
\]
Hence, $1-\frac{ x_{J}^\top \Lambda^{0.5}_{J,J} x_{J}}{\sqrt{x^\top \Lambda x}}\geq1-x^\top\frac{\Lambda^{0.5} x}{\|\Lambda^{0.5}x\|_2}$.
Since $\Lambda$ is positive definite and $x$ is a unit vector, $\frac{\Lambda^{0.5} x}{\|\Lambda^{0.5}x\|_2}$ is also a unit vector, and, therefore, $1-x^\top\frac{\Lambda^{0.5} x}{\|\Lambda^{0.5}x\|_2}\geq0$.
\halmos
\endproof

The following lemma in \citet{auer2002using} is also needed for proving our Lemma \ref{lm:Xinshang_lemma}.
\begin{lemma}[{\citet{auer2002using}, Lemma 19}]
\label{lm:auerA}
Let $\lambda_1\geq\cdots\geq\lambda_d\geq0$.
The eigenvalues $\nu_1,\ldots,\nu_d$ of a matrix $\Delta(\lambda_1,\ldots,\lambda_d)+zz^\top$ with $\|z\|_2\leq1$ can be arranged such that there are $y_{h,j}\geq 0$, $1\leq h < j\leq d$, and the following holds:
\begin{align*}
&\nu_j\geq\lambda_j,\\
&\nu_j=\lambda_j+z_j^2-\sum_{h=1}^{j-1}y_{h,j}+\sum_{h=j+1}^dy_{j,h}\\
&\sum_{h=1}^{j-1}y_{h,j}\leq z_j^2\\
&\sum_{h=j+1}^dy_{j,h}\leq\nu_j-\lambda_j,\\
&\sum_{j=1}^d\nu_j=\sum_{j=1}^d\lambda_j+\|z\|^2_2.
\end{align*}
\end{lemma}

\begin{lemma} \label{lm:key1}
Let $A \in \bR^{d\times d}$ be any symmetric positive definite matrix, and $u \in \bR^d$ be any vector.
Let  $\lambda_1, \lambda_2, \ldots, \lambda_d$ be the eigenvalues of $A$, and $\nu_1,\nu_2,\ldots, \nu_d$ be the eigenvalues of $A + u u^\top$.
We must have for any $x \in \bR^d$ such that $\|x\|_2 = 1$,
\[
\sqrt{x^\top A\inv x} - \sqrt{x^\top (A + uu^\top)\inv x} \leq  \sum_{i=1}^d \frac{2}{\sqrt{\lambda_i}} - \sum_{i=1}^d \frac{2}{\sqrt{\nu_i}}.
\]
\end{lemma}
\proof{Proof.}
Define matrix $A(s) = A + s \cdot uu^\top$ for all $s \in [0,1]$.
Let $\ls_1, \ls_2,...,\ls_d$ be the eigenvalues of $A(s)$, for all $s \in [0,1]$.
Without loss of generality, suppose $\ls_1 \geq \ls_2 \geq \cdots \geq \ls_d$.
Define function $f(s) = \sqrt{x^\top A(s)\inv x}$ for all $s \in [0,1]$.
The theorem can be equivalently written as
\[
f(0) - f(1) = - \int_0^1 f'(s) ds \leq \sum_{i=1}^d \frac{2}{\sqrt{\lambda_i^{(0)}}} - \sum_{i=1}^d \frac{2}{\sqrt{\lambda_i^{(1)}}},
\]
where $f'(s)$ is the derivative of $f(s)$.
It suffices to prove that, for all $s \in [0,1]$,
\[
\sum_{i=1}^d \frac{2}{\sqrt{\ls_i}}\leq f'(s) \leq 0.
\]
 
We have
\begin{align*}
f'(s) = & \frac{d}{ds} \sqrt{x^\top A(s)\inv x}\\
= & \frac{ 0.5 }{ \sqrt{x^\top A(s)\inv x} } \cdot \frac{d}{ds} \left[ x^\top A(s)\inv x\right]\\
= & \frac{ 0.5 }{ \sqrt{x^\top A(s)\inv x} } \cdot \lim_{\delta \to 0} \frac{x^\top \left(A(s) +\delta \cdot  u u^\top\right) \inv x -  x^\top A(s)\inv x}{\delta}\\
= & \frac{ 0.5 }{ \sqrt{x^\top A(s)\inv x} } \cdot \lim_{\delta \to 0} \frac{x^\top \left(A(s)\inv - \frac{\delta A(s)\inv u u^\top A(s)\inv}{1 + \delta u^\top A(s)\inv u} \right)  x-   x^\top A(s)\inv x}{\delta} \tag{by the Sherman-Morrison formula}\\
= & - \frac{ 0.5 }{ \sqrt{x^\top A(s)\inv x} } \cdot \lim_{\delta \to 0} \frac{(x^\top A(s)\inv u)^2}{1 + \delta u^\top A(s)\inv u} \\
= &  -0.5 \frac{(x^\top A(s)\inv u)^2}{\sqrt{x^\top A(s)\inv x}}.
\end{align*}
Thus, $f'(s) \leq 0$.

Let $A(s)\inv = U(s)\Lambda(s)\inv U(s)^\top$ be the eigendecomposition of $A(s)\inv$.
Let $\tilde x = U^\top x$ and $\tilde u = U^\top u$.
Using Lemma \ref{lm:key0}, we obtain
\begin{align}
 f'(s) & =  -0.5 \frac{(x^\top A(s)\inv u)^2}{\sqrt{x^\top A(s)\inv x}} \nonumber\\
 &  =  -0.5 \frac{(\tilde x^\top \Lambda(s)\inv \tilde u)^2}{\sqrt{\tilde x^\top \Lambda(s)\inv \tilde x}} \nonumber\\
 & \geq -  0.5 \sum_{i=1}^d  \tilde u_i^2 (\ls_i)^{-1.5}.\label{eq:key1bound1}
 \end{align}

For small $ \delta > 0$, we can write $$A(s+\delta) = A(s) + \delta u u^\top = U(s) (\Lambda(s) + \delta \tilde u \tilde u^\top) U(s)^\top.$$
We use Lemma \ref{lm:auerA} to obtain the following properties
\begin{itemize}
\item There are non-negative numbers $y_{h,i}(s,\delta)$, for $i =1,2,...,d$, $h=1,2,...,d$ and $h\not=i$, such that
\begin{equation}\label{eq:key1a}
 \lsd_i = \ls_i + \delta \tilde u_i^2 - \sum_{h=1}^{i-1} y_{h,i}(s,\delta) + \sum_{h=i+1}^d y_{i,h}(s,\delta).
 \end{equation}
\item For any $\ls_h > \ls_i + \delta \|\tilde u\|_2^2$,
\begin{equation}\label{eq:key1b}
y_{h,i}(s,\delta) \leq \frac{\delta^2 \tilde u_i^2 \tilde u_h^2}{ \ls_h - \ls_i - \delta \|\tilde u\|_2^2 }.
\end{equation}
\item For all $i=1,2,...,d$,
\begin{equation}\label{eq:key1c}
 \sum_{h=i+1}^d y_{i,h}(s,\delta) \leq \lsd_i - \ls_i.
 \end{equation}
\end{itemize}

Suppose $\delta$ is small enough so that $\delta \|\tilde u\|_2^2 \leq \sqrt{\delta} \ls_d$. For any $\ls_h > \ls_i + \sqrt{\delta} \ls_d$, \eqref{eq:key1b} further gives
\begin{equation}\label{eq:key1d}
 y_{h,i}(s,\delta) \leq \frac{\delta^2 \tilde u_i^2 \tilde u_h^2}{ \ls_h - \ls_i - \delta \|\tilde u\|_2^2 } \leq \frac{\delta^2 \tilde u_i^2 \tilde u_h^2}{ \sqrt{\delta} \ls_d - \delta \|\tilde u\|_2^2 }. 
 \end{equation}

For $\ls_h \leq \ls_i + \sqrt{\delta} \ls_d$,
\begin{align}
& \sum_{i=1}^d  (\ls_i)^{-1.5} \sum_{h < i: \ls_h \leq \ls_i + \sqrt{\delta} \ls_d} y_{h,i}(s,\delta)\nonumber\\
= & \sum_{i=1}^d  \sum_{h < i: \ls_h \leq \ls_i + \sqrt{\delta} \ls_d} y_{h,i}(s,\delta)\frac{(\ls_h)^{1.5}}{(\ls_i)^{1.5}}  (\ls_h)^{-1.5} \nonumber\\
\leq & \sum_{i=1}^d  \sum_{h < i: \ls_h \leq \ls_i + \sqrt{\delta} \ls_d} y_{h,i}(s,\delta)\left( \frac{\ls_i + \sqrt{\delta} \ls_d}{\ls_i}\right)^{1.5}  (\ls_h)^{-1.5} \nonumber\\
\leq & \sum_{i=1}^d  \sum_{h < i: \ls_h \leq \ls_i + \sqrt{\delta} \ls_d} y_{h,i}(s,\delta)\left( 1 + \sqrt{\delta} \right)^{1.5}  (\ls_h)^{-1.5} \nonumber\\
\leq &\left( 1 + \sqrt{\delta} \right)^{1.5}  \sum_{h=1}^d  \sum_{i=h+1}^d y_{h,i}(s,\delta)  (\ls_h)^{-1.5} \nonumber\\
\leq &\left( 1 + \sqrt{\delta} \right)^{1.5}  \sum_{h=1}^d  (\lsd_h - \ls_h) (\ls_h)^{-1.5}. \label{eq:key1e}
\end{align}
The last inequality follows from \eqref{eq:key1c}.

We use \eqref{eq:key1d} and \eqref{eq:key1e} to obtain
\begin{align}
&\lim_{\delta \to 0}  \frac{1}{\delta} \sum_{i=1}^d  (\ls_i)^{-1.5} \sum_{h=1}^{i-1} y_{h,i}(s,\delta)\nonumber\\
=&\lim_{\delta \to 0}  \frac{1}{\delta} \sum_{i=1}^d  (\ls_i)^{-1.5} \left[ \sum_{h < i: \ls_h > \ls_i + \sqrt{\delta} \ls_d} y_{h,i}(s,\delta)+  \sum_{h < i: \ls_h \leq \ls_i + \sqrt{\delta} \ls_d} y_{h,i}(s,\delta) \right]\nonumber\\
\leq & \lim_{\delta \to 0}  \frac{1}{\delta} \sum_{i=1}^d  (\ls_i)^{-1.5} \left[\sum_{h < i: \ls_h > \ls_i + \sqrt{\delta} \ls_d} \frac{\delta^2 \tilde u_i^2 \tilde u_h^2}{ \sqrt{\delta} \ls_d - \delta \|\tilde u\|_2^2 }+  \sum_{h < i: \ls_h \leq \ls_i + \sqrt{\delta} \ls_d} y_{h,i}(s,\delta) \right]\nonumber\\
= & \lim_{\delta \to 0}  \frac{1}{\delta} \sum_{i=1}^d  (\ls_i)^{-1.5}  \sum_{h < i: \ls_h \leq \ls_i + \sqrt{\delta} \ls_d} y_{h,i}(s,\delta) \nonumber\\
\leq & \lim_{\delta \to 0}  \frac{1}{\delta} \left( 1 + \sqrt{\delta} \right)^{1.5}  \sum_{h=1}^d  (\lsd_h - \ls_h) (\ls_h)^{-1.5} \nonumber\\
= & \sum_{h=1}^d  (\ls_h)^{-1.5} \frac{d}{ds}\ls_h\nonumber\\
= & -2 \sum_{h=1}^d \frac{d}{ds}\frac{1}{\sqrt{\ls_h}}. \label{eq:key1f}
 \end{align}

Hence,
\begin{align*}
&\frac{d}{ds}  \sum_{i=1}^d \frac{1}{\sqrt{\ls_i}}\\
= & - 0.5 \sum_{i=1}^d (\ls_i)^{-1.5} \frac{d}{ds}\ls_i\\
= & - 0.5 \sum_{i=1}^d (\ls_i)^{-1.5} \lim_{\delta \to 0} \frac{\lsd_i - \ls_i}{\delta}\\
= & - 0.5 \sum_{i=1}^d (\ls_i)^{-1.5} \lim_{\delta \to 0} \frac{\ls_i + \delta \tilde u_i^2 - \sum_{h=1}^{i-1} y_{h,i}(s,\delta) + \sum_{h=i+1}^d y_{i,h}(s,\delta) - \ls_i}{\delta}\\
= & - 0.5 \sum_{i=1}^d (\ls_i)^{-1.5} \left( \tilde u_i^2 + \lim_{\delta \to 0} \frac{- \sum_{h=1}^{i-1} y_{h,i}(s,\delta) + \sum_{h=i+1}^d y_{i,h}(s,\delta)}{\delta}\right) \\
\leq & - 0.5 \sum_{i=1}^d (\ls_i)^{-1.5} \left( \tilde u_i^2 + \lim_{\delta \to 0} \frac{- \sum_{h=1}^{i-1} y_{h,i}(s,\delta)}{\delta}\right) \\
= & - 0.5 \sum_{i=1}^d \tilde u_i^2 (\ls_i)^{-1.5} + 0.5  \lim_{\delta \to 0} \frac{\sum_{i=1}^d  (\ls_i)^{-1.5} \sum_{h=1}^{i-1} y_{h,i}(s,\delta)}{\delta}\\
\leq & - 0.5 \sum_{i=1}^d \tilde u_i^2 (\ls_i)^{-1.5} + 0.5 \cdot \left( -2 \sum_{i=1}^d \frac{d}{ds}\frac{1}{\sqrt{\ls_i}}\right),
\end{align*}
where the last inequality follows from \eqref{eq:key1f}.
The above result leads to
\[ \frac{d}{ds}  \sum_{i=1}^d \frac{1}{\sqrt{\ls_i}} \leq - \frac{1}{4} \sum_{i=1}^d \tilde u_i^2 (\ls_i)^{-1.5} .\]
Combining this with \eqref{eq:key1bound1}, we obtain
\[ f'(s) \geq - 0.5 \sum_{i=1}^d \tilde u_i^2 (\ls_i)^{-1.5}  \geq \frac{d}{ds}  \sum_{i=1}^d \frac{2}{\sqrt{\ls_i}},\]
which proves the lemma.
\halmos
\endproof

Now we present the proof of Lemma~\ref{lm:Xinshang_lemma}.

\PROOF{Lemma \ref{lm:Xinshang_lemma}}
\proof{Proof.}

Define $A_k = A + \sum_{k'=1}^k u_{k'}u_{k'}^\top$ for all $k=0,1,2,...,l$. Let $\lambda_1^{(k)}, \lambda_2^{(k)},...,\lambda_d^{(k)}$ be the eigenvalues of $A_k$, for all $k=0,1,2,...,l$.

 Since $\sqrt{x^\top A_{k-1}\inv x} \geq \sqrt{x^\top A_k\inv x}$ for all $k=1,2,...,l$, we can use Lemma \ref{lm:key1} to obtain
\begin{align*}
&\sqrt{x^\top A\inv x} - \sqrt{x^\top (A +  \sum_{k=1}^l u_k u_k^\top)\inv x}\\
= & \sum_{k=1}^l \left[ \sqrt{x^\top A_{k-1}\inv x} - \sqrt{x^\top A_k\inv x} \right]\\
\leq & \sum_{k=1}^l \left[ \sum_{i=1}^d \frac{1}{\sqrt{\lambda_i^{(k-1)}}} - \sum_{i=1}^d \frac{1}{\sqrt{\lambda_i^{(k)}}} \right]\\
= & \sum_{i=1}^d \frac{2}{\sqrt{\lambda_i^{(0)}}} - \sum_{i=1}^d \frac{2}{\sqrt{\lambda_i^{(k)}}}\\
= & \sum_{i=1}^d \frac{2}{\sqrt{\lambda_i}} - \sum_{i=1}^d \frac{2}{\sqrt{\nu_i}}.
\end{align*}

\halmos
\endproof